%% file: persist3f.tex
\documentclass{article}
\usepackage{graphicx,amsmath,amsfonts,amssymb,color,amsthm,psfrag}
\usepackage{a4wide,mathrsfs,ifthen}
\usepackage{amsfonts}
\usepackage{amsmath}
\usepackage{amssymb}
\include{sdefsf}

\newcommand{\Z}{Z^{c_0,c_1}}
\newcommand{\N}{N^{c_0,c_1}}
\newcommand{\z}{Z^{c_1,c_0}} 
\newcommand{\n}{N^{c_1,c_0}}

\newcommand{\NN}{N^{c'_0,c'_1}}
\newcommand{\zz}{Z^{c'_1,c'_0}}
\newcommand{\nn}{N^{c'_1,c'_0}}
\title{From persistent random walk to the telegraph noise}
\author{Samuel Herrmann and Pierre Vallois\\
  Institut de Math\'ematiques Elie Cartan - UMR 7502\\
  Nancy-Universit\'e, CNRS, INRIA\\
  B.P. 239, 54506 Vandoeuvre-l\`es-Nancy Cedex, France\\
  \{herrmann,vallois\}@iecn.u-nancy.fr
}
\newboolean{complet}
\setboolean{complet}{false}
\title{From persistent random walks to the telegraph noise}
\begin{document}

\maketitle

\abstract{\small
We study a family of memory-based persistent random walks and we prove weak 
convergences after space-time rescaling. The limit processes are not only 
Brownian motions with drift. We have obtained a continuous but non-Markov 
process $(Z_t)$ which can be easely expressed in terms of a counting process 
$(N_t)$. In a particular case the counting process is a Poisson process, and 
$(Z_t)$ permits to represent the solution of the telegraph equation. We study 
in detail the Markov process $((Z_t,N_t); \ t\ge 0)$.
}
\section{The setting of persistent random walks.}
\label{section_not}
    1) The simplest way to present and define a persistent random walk
    with value in $\Zset$ is to introduce the process of its increments
    $(Y_t, \ t\in\Nset)$. In the classical symmetric random walk case,
    this process is just a sequence of independent random variables
    satisfying $\P(Y_t=1)=\P(Y_t=-1)=\frac{1}{2}$ for any $t\ge 0$.
    Here we shall introduce some short range memory in these increments
    in order to create the persistence phenomenon. Namely $(Y_t)$ is a
    $\{-1,1\}$-valued Markov chain: the law of $Y_{t+1}$ given 
    $\mathcal{F}_t=\sigma(Y_0, Y_1,\ldots, Y_t)$ depends only
    on the value of $Y_t$. This dependence is represented by the
    transition probability $\pi(x,y)=\P(Y_{t+1}=y\vert Y_{t}=x)$ with
    $(x,y)\in\{-1,1\}^2$:
\begin{eqnarray*}
        \pi=\left(\begin{array}{cc} 1-\alpha & \alpha\\
        \beta & 1-\beta\end{array}\right)\quad\quad 0<\alpha<1,\quad
        0<\beta<1.
\end{eqnarray*}
The persistent random walk is the corresponding process of partial sums:
\begin{equation}
\label{25*1} X_t=\sum_{i=0}^t Y_i\quad\mbox{with }\quad X_0=Y_0=1\
\mbox{or }\ -1.
\end{equation}
    Let us discuss two particular cases:
\begin{itemize}
    \item If $\alpha+\beta=1$, then increments  are independent and therefore 
    the short range memory disappears. $(X_t,\ t\in\Nset)$ is a classical Bernoulli 
    random walk. 
    \item The symmetric case $\alpha=\beta$ was
    historically suggested by F\"{u}rth \cite{furthbook20} and precisely
    defined by Taylor \cite{taylor}. Goldstein \cite{goldstein50}
    developed the calculation of the random walk law and clarified the
    link between this process and the so-called telegraph equation.
    Some nice presentation of these results can be found in Weiss'
    book \cite{weissbook94} and  \cite{weiss02}. This particular short
    memory process is often called either \emph{persistent or correlated
    random walk} or \emph{Kac walks} (see, for instance, \cite{eckstein00}). 
    An interesting presentation of different limiting distributions for this 
    correlated random walk has been given by Renshaw and Henderson \cite{renshaw81}.
\end{itemize}
    2) Recently, Vallois and Tapiero \cite{vallois07} studied the influence of 
    the persistence phenomenon on the first and second moments of a counting process
    whose increments takes their values in $\{0,1\}$ instead of $\{-1,1\}$. 
    They obtained some nearly linear behaviour for the expectation. 
    Using the transformation $y\to 2y-1$, it is easy to deduce that, 
    in our setting, we have:

\begin{equation}
\label{11b}
\E_{-1}[X_t]:=\E[X_t\vert
X_0=Y_0=-1]=\frac{\alpha-\beta}{1-\rho}\,(t+1)-\frac{2\alpha}{(1-\rho)^2}\,
(1-\rho^{t+1}).
\end{equation}
\begin{equation}
\label{11c}
\E_{+1}[X_t]:=\E[X_t\vert
X_0=Y_0=+1]=\frac{\alpha-\beta}{1-\rho}\,(t+1)-\frac{2\beta}{(1-\rho)^2}\,
(1-\rho^{t+1}).
\end{equation}
An application to insurance has been given in \cite{vallois08}.\\
It is actually possible to determine the moment generating function 
(see Proposition \ref{a+a-1} in Section \ref{section preuve2}).
\[
\Phi(\lambda,t)=\E[\lambda^{X_t}],\quad(\lambda\in\Rset_+^*).
\]
However it seems difficult to invert this transformation; i.e. to give 
the law of $X_t$.\\
3) This leads us to investigate limit distributions. It is well-known 
that the correctly normalized symmetric random walk converges towards 
the Brownian motion. Let us define the time and space normalizations.
Let $\alpha_0$ and $\beta_0$ denote two real numbers satisfying:
\begin{equation}
\label{*1**}
0\le\alpha_0\le 1,\quad 0\le\beta_0\le 1.
\end{equation}
Let $\Delta_x$ be a positive small parameter so that:
\begin{equation}
\label{*2**}
0\le \alpha_0+c_0\Delta_x\le 1,\quad 0\le \beta_0+c_1\Delta_x\le 1,
\end{equation}
where $c_0$ and $c_1$ belong to $\Rset$ (see in subsection \ref{***} 
the allowed range of parameters).\\
Let $(Y_t,\, t\in\Nset)$ be a Markov chain whose transition probabilities 
are given by the matrix:
\begin{eqnarray}
\label{1-4B}
\pi^\Delta=\left(\begin{array}{cc}
1-\alpha_0-c_0\Delta_x & \alpha_0+c_0\Delta_x\\
\beta_0+c_1\Delta_x & 1-\beta_0-c_1\Delta_x
      \end{array}\right).
\end{eqnarray}
Let $(X_t,\ t\in\Nset)$ be the random walk associated with $(Y_t)$ 
(cf. \eqref{25*1}). 
Define the normalized random walk $(Z^\Delta_s, \, s\in\Delta_t\Nset)$ 
by the relation:
\begin{equation}\label{scal}
 Z^\Delta_s=\Delta_x X_{s/\Delta_t}, \quad (\Delta_t>0, \ \Delta_x>0).
\end{equation}
 Set $(\tilde{Z}^\Delta_s,\, s\ge 0)$ the continuous time process obtained 
 by linear interpolation of $(Z_s^\Delta)$.\\ We introduce two essential parameters:
\begin{equation}
\label{double1}
\rho_0=1-\alpha_0-\beta_0\quad\mbox{(the asymmetry coefficient),}
\end{equation}
\begin{equation}
\label{double2}
\eta_0=\beta_0-\alpha_0.
\end{equation}
In this paper, we will aim at showing the existence of a normalization 
(i.e. to express $\Delta_t$ in terms of $\Delta_x$) which depends on $\alpha_0$, 
$\beta_0$, so that $(\tilde{Z}^\Delta_s)$ converges in distribution, 
as $\Delta_x\to 0$.\\
Our main results and the organization of the paper will be given in Section 
\ref{results}.
\section{The main results}\label{results}
\mathversion{bold}
\subsection{Case : $\rho_0=1$}
\label{section2.1}
\mathversion{normal}
Obviously $\rho_0=1$ implies that $\alpha_0=\beta_0=0$, and the transition 
probabilities matrix is given as
\begin{eqnarray*}
\pi^\Delta=\left(\begin{array}{cc}
1-c_0\Delta_x & c_0\Delta_x\\
c_1\Delta_x & 1-c_1\Delta_x
      \end{array}\right)\quad (c_0,c_1>0).
\end{eqnarray*}
In order to describe the limiting process, we introduce a sequence of 
independent identically exponentially distributed random variables $(e_n,n\ge 1)$ 
with $\E[e_n]=1$. We construct the following counting process:
\begin{equation}
\label{**1.1}
N_t^{c_0,c_1}=\sum_{k\ge 1}1_{\{\lambda_1 e_1+\lambda_2 e_2+\ldots+\lambda_k e_k\le t\}},
\end{equation}
where 
\begin{eqnarray}
\label{.*.}
\lambda_k=\left\{\begin{array}{l}
1/c_0\quad \mbox{if}\ k\ \mbox{is odd}\\
1/c_1\quad \mbox{otherwise.}
\end{array}\right.
\end{eqnarray}
Finally we define
\begin{equation}
\label{def_1.2}
\Z_t=\int_0^t (-1)^{\N_u}du.
\end{equation}
For simplicity of notations, in the symmetric case (i.e. $c_0=c_1$), $N^{c_0}_t$ 
(resp. $Z^{c_0}_t$) will stand for $N^{c_0,c_0}_t$ (resp. $Z^{c_0,c_0}_t$). The 
process $(Z_t^{c_0})$ has been introduced by Stroock (in \cite{stroock} p. 37). 
It is possible to show that if we rescale $(Z_t^{c_0})$, this process converges 
in distribution to the standard Brownian motion. This property has been widely 
generalized. For instance Bardina and Jolis \cite{bardina00} have given weak 
approximation of the Brownian sheet from a Poisson process in the plane.
\begin{thm}
\label{cas1} Let $\Delta_x=\Delta_t$ and $Y_0=X_0=-1$. Then the
interpolated persistent random walk $(\tilde{Z}^\Delta_s,\ s\ge
0)$ converges in distribution, as $\Delta_x\to 0$, to the process $(-\Z_s,\ s\ge 0)$.\\
In particular if $c_0=c_1$, then $(N_u^{c_0})$ is the Poisson process with 
parameter $c_0$.\\
If $Y_0=X_0=1$ then the
interpolated persistent random walk $(\tilde{Z}^\Delta_s,\ s\ge
0)$ converges in distribution, as $\Delta_x\to 0$, to the process $(\z_s,\ s\ge 0)$.
\end{thm}
\begin{proof} See Section \ref{sectionpreuve1}.\end{proof}
Next, in Section \ref{properties_*}, we investigate the process $(\Z_t,\N_t;\ t\ge 0)$. 
In particular we prove that it is Markov, we determine its semigroup and the law 
of $(\Z_t,\N_t)$, $t$ being fixed. This permits to prove, when $c_0=c_1$, the 
well-known relation (cf. \cite{weiss02}, \cite{eckstein00}, \cite{goldstein50}, 
\cite{griego71}) between the solutions of the wave equation and the telegraph 
equation. For this reason the process $(\Z_t)$ will be called the integrated 
telegraph noise (ITN for short).\\
We emphasize that our approach based on stochastic processes gives a better 
understanding of analytical properties.\\
We will give in Section \ref{extensions} below two extensions of Theorem \ref{cas1} 
to the cases where $(Y_t)$ is\\
1) a Markov chain which takes its values in $\{y_1,\ldots,y_k\}$,\\
2) a Markov chain with order $2$ and valued in $\{-1,1\}$.
\mathversion{bold}
\subsection{Case : $\rho_0\neq1$}
\label{section2.2}
\mathversion{normal} In this case, the limit process is Markov. We shall prove 
two kind of convergence results. The
first one corresponds to the law of large numbers
and the second one looks like functional central limit theorem. \\
Recall that $(\tilde{Z}^\Delta_t,\, t\ge 0)$ is the linear interpolation of 
$(Z_t^\Delta)$ and $\rho_0$ (resp. $\eta_0$) has been defined by \eqref{double1} 
(resp. \eqref{double2}).
\begin{thm}
\label{conv1}
1) Suppose that  
$r\Delta_t=\Delta_x$ with $r>0$. Then $\tilde{Z}^\Delta_t$ converges to the 
deterministic limit $-\frac{rt\eta_0}{1-\rho_0}$ when $\Delta_x\to 0$.\\
2) Suppose that $r\Delta_t=\Delta_x^2$ with $r>0$, then the process
$(\xi^\Delta_t,\ t\ge 0)$ defined by
\[
\xi^{\Delta}_t=\tilde{Z}^\Delta_t+\frac{t\sqrt{r}\eta_0}{(1-\rho_0)\sqrt{\Delta_t}}
\]
converges in distribution to the process $(\xi^0_t,\ t\ge 0)$, as $\Delta_x\to 0$, 
where
\begin{equation}
\label{juil3}
\xi^0_t=2r\Big(\frac{-\overline{\tau}}{1-\rho_0}+\frac{\eta_0
\tau}{(1-\rho_0)^2}\Big)t+\sqrt{\frac{r(1+\rho_0)}{1-\rho_0}\Big(1-
\frac{\eta_0^2}{(1-\rho_0)^2}\Big)}W_t,
\end{equation}
($W_t$, $t\ge 0$) is a one-dimensional Brownian motion, $
\tau=(c_0+c_1)/2$ and $\overline{\tau}=(c_1-c_0)/2$.
\end{thm}
\begin{proof} See Section \ref{section preuve2}.\end{proof}
Gruber and Schweizer have proved in \cite{gruber06} a weak convergence result 
for a large class of generalized correlated random walks. However these results 
and ours can be only compared in the case
$\alpha_0=\beta_0$.\\[5pt]
Note that 
\[
1-\frac{\eta_0^2}{(1-\rho_0)^2}=0\Longleftrightarrow \alpha_0=0\quad\mbox{or}\quad\beta_0=0.
\]
Suppose for instance that $\alpha_0=0$. Then $\beta_0,c_0>0$ and 
\[
\xi^{\Delta}_t=\tilde{Z}^{\Delta}_t+\frac{t\sqrt{r}}{\sqrt{\Delta_t}}\quad\mbox{and}\quad \xi^0_t=\frac{2rc_0}{\beta_0}\, t.
\]
Obviously, the diffusion coefficient of $(\xi^0_t)$ can also cancel when $\rho_0=-1$.\\
Since $\rho_0=-1\Longleftrightarrow\alpha_0=\beta_0=1$, then $c_0,c_1<0$ and 
\[
\xi^{\Delta}_t=\tilde{Z}^{\Delta}_t\quad\mbox{and}\ \xi^0_t=-r\overline{\tau}t.
\]
This shows that, in the symmetric case (i.e. $c_0=c_1$), we have $\xi^0_t=0$. 
This means that the normalization is not the right one since the limit is null. 
Changing the rescaling we can obtain a non-trivial limit.
\begin{prop}
\label{cas_limit}
Suppose $\alpha_0=\beta_0=1$, $c_0=c_1<0$ and $r\Delta_t=\Delta_x^3$ with $r>0$. \\
The interpolated persistent walk $(\tilde{Z}^\Delta_t,\ t\ge 0)$
converges in law, as
$\Delta_x\to 0$, to $(\sqrt{-r c_0}W_t,\, t\ge 0)$ where $(W_t)$ is a standard 
Brownian motion. 
\end{prop}
\begin{proof} See subsection \ref{scas_limit}\end{proof}
\subsection{Organization of the paper}
The third section presents few properties of the process $(\Z_t,\, t\ge 0)$ 
which has been defined by \eqref{def_1.2}. Theorem \ref{cas1} will be proven 
in Section \ref{sectionpreuve1}. Section \ref{extensions} will be devoted to 
two extensions of Theorem \ref{cas1}. In subsection \ref{mgf} we determine the 
generating function of $X_t$ (recall that $X_t$ has been defined by \eqref{25*1}). 
This is the main tool which permits to prove Theorem \ref{conv1} and Proposition 
\ref{cas_limit} (see subsections \ref{***} and \ref{scas_limit}).
\section{Properties of the integrated telegraph noise}\label{properties_*}
The aim of this section is to study the two dimensional process $(\Z_t,\
\N_t;\ t\ge 0)$ introduced in \eqref{.*.} and \eqref{def_1.2}. In the particular 
symmetric case $c_0=c_1$, the study is simpler since the process 
$(N^{c_0}_t,\ t\ge0)$ is a Poisson process with rate $c_0$
($\E(N^{c_0}_t)=c_0 t$) and $N_0^{c_0}=0$. However we shall study the general case. \\

First, we determine in Proposition \ref{loi__1} the conditional density of 
$\Z_t$ given $\N_t=n$. As a by product we obtain the distribution of $\Z_t$ 
(see Proposition \ref{bessel__*}). Second, we prove in Proposition \ref{markk} 
that $(\Z_t,\N_t,\, t\ge 0)$ is Markov and we determine its semi-group. 
We conclude this section by showing that the solution of the telegraph equation 
can be expressed in terms of the associated wave equation and 
$(Z_t^{c_0,c_0})_{t\ge 0}$. For this reason, $(\Z_t)_{t\ge 0}$ will be called 
the integrated telegraph noise (ITN for short).
Recall that:
\begin{equation}
\label{constantes_c}
\tau=\frac{c_0+c_1}{2},\quad\quad \overline{\tau}=\frac{c_1-c_0}{2}.
\end{equation}
\begin{prop}
\label{loi__1}
1) $\P(\N_t=0)=e^{-tc_0}$ and given $\N_t=0$, we have $\Z_t=t$.\\
2) The counting process takes even values with probability:
\begin{equation}
\label{pair__*}
\P(\N_t=2k)=\frac{(c_0c_1)^k\alpha_k(t)}{2^{2k}k!(k-1)!}\, e^{-\tau t}\quad\mbox{with}\
\alpha_k(t)=\int_{-t}^t (t-z)^{k-1}(t+z)^{k} e^{\overline{\tau}z}dz,
\end{equation}
and the conditional distribution of $\Z_t$ is given by
\begin{equation}
\label{cond_pair}
\P(\Z_t\in dz\vert \N_t=2k)=\frac{1}{\alpha_k(t)}(t-z)^{k-1}(t+z)^k e^{\overline{\tau}z}1_{[-t,t]}(z)\quad (k\ge 1).
\end{equation}
3) The counting process takes odd values with probability:
\begin{equation}
\label{impair__*}
\P(\N_t=2k+1)=\frac{c_0^{k+1}c_1^k\tilde{\alpha}_k(t)}{2^{2k+1}(k!)^2}\, e^{-\tau t}\quad\mbox{with}\
\tilde{\alpha}_k(t)=\int_{-t}^t (t-z)^{k}(t+z)^{k} e^{\overline{\tau}z}dz,
\end{equation}
and the conditional distribution of $\Z_t$ is given by
\begin{equation}
\label{cond_impair}
\P(\Z_t\in dz\vert \N_t=2k+1)=\frac{1}{\tilde{\alpha}_k(t)}(t-z)^{k}(t+z)^k e^{\overline{\tau}z}1_{[-t,t]}(z)\quad (k\ge 0).
\end{equation}
\end{prop}
\begin{cor}
In the particular symmetric case $c_0=c_1$, the conditional density function
of $Z^{c_0}_t$ given $N^{c_0}_t=n$ is the centered beta density, i.e.
\begin{equation}\label{eq:loi cond1_*}
\mbox{for}\ n=2k,\ k\in\mathbb{N^*}:\quad
f_n(t,z)=\chi_{2k}\frac{(t+z)^k(t-z)^{k-1}}{t^{2k}}1_{[-t,t
]}(z),
\end{equation}
\begin{equation}\label{eq:loi cond2_*}
\mbox{for}\ n=2k+1,\ k\in\mathbb{N}:\quad
f_n(t,z)=\chi_{2k+1}\frac{(t+z)^k(t-z)^{k}}{t^{2k+1}}1_{[-t, t
]}(z),
\end{equation}
with
\begin{eqnarray*}
\chi_{2k+1}=\chi_{2k+2}=\frac{1}{2^{2k+1}B(k+1,k+1)}=\frac{(2k+1)!}{2^{2k+1}(k!)^2}\quad
(k\ge 0),
\end{eqnarray*}
($B$ is the beta function (first Euler function):
$B(r,s)=\frac{\Gamma(r)\Gamma(s)}{\Gamma(r+s)}$).
\end{cor}
\noindent
{\it Proof of Proposition \ref{loi__1}}.
Associated with $n\ge 0$ and a bounded continuous function $f$, we define
\[
\Delta_n(f)=\E\Big[f(\Z_t)1_{\{\N_t=n\}}\Big].
\]
{\bf a)} When $n=0$, we obtain
\[
\Delta_0(f)=\E\Big[ f(\Z_t)1_{\{t<\lambda_1 e_1\}} \Big].
\]
If $t<\lambda_1 e_1$, then $\Z_t=t$ and
\[
\Delta_0(f)=f(t)\P(t<\lambda_1 e_1)=f(t)e^{-tc_0}.
\]
{\bf b)} When $n\ge 1$, using \eqref{**1.1} we obtain
\[
\Delta_n(f)=\E\Big[f(\Z_t)1_{\{\lambda_1 e_1+\ldots+\lambda_n e_n \le t< \lambda_1 e_1+\ldots +\lambda_{n+1}e_{n+1}\}}\Big].
\]
If $\lambda_1 e_1+\ldots+\lambda_n e_n \le t< \lambda_1 e_1+\ldots +\lambda_{n+1}e_{n+1}$ then
\begin{eqnarray*}
\Z_t&=&\int_0^{\lambda_1 e_1}(-1)^0du+\int_{\lambda_1 e_1}^{\lambda_1 e_1+\lambda_2 e_2}(-1) du+\ldots+\int_{\lambda_1 e_1+\ldots+\lambda_{n-1}e_{n-1}}^{\lambda_1 e_1+\ldots+\lambda_{n}e_{n}}(-1)^{n-1}du\\
&+&\int_{\lambda_1 e_1+\ldots+\lambda_{n}e_{n}}^t (-1)^n du.
\end{eqnarray*}
Hence
\begin{equation}
\label{resum__*}
\Z_t=\lambda_1 e_1-\lambda_2e_2+\lambda_3e_3+\ldots +(-1)^{n-1}\lambda_n e_n +(-1)^n(t-\lambda_1 e_1-\ldots-\lambda_n e_n).
\end{equation}
{\bf c)} Evaluation of $\Delta_{2k}(f)$, $k\ge 1$.\\
We introduce two sequences of random variables associated with $(e_n)$:
\begin{equation}
\label{pair-impair}
\xi_k^e=e_2+\ldots+e_{2k},\quad \xi^o_k=e_1+\ldots+e_{2k-1},\quad (k\ge 1).
\end{equation}
By \eqref{resum__*}, \eqref{.*.} and \eqref{pair-impair} we obtain the simpler 
expression
\[
\Delta_{2k}(f)=\E\Big[ f(t-2\xi_k^e/c_1)1_{\{ \xi^o_k/c_0+\xi_k^e/c_1\le t< \xi^o_k/c_0+\xi_k^e/c_1+ e_{2k+1}/c_0\}} \Big].
\]
Note that from our assumptions,  $\xi^e_k$, $\xi^o_k$ and $e_{2k+1}$ are 
independent r.v.'s,  $\xi^o_k$ and $\xi^e_k$ are both  gamma distributed 
with parameter $k$. Consequently: 
\begin{eqnarray*}
\Delta_{2k}(f)&=&\frac{1}{((k-1)!)^2}\int_{D_t }\exp\{-c_0(t-y/c_0-x/c_1)\}
f(t-2 x/c_1)e^{-x-y}x^{k-1}y^{k-1}dx\,dy\\
&=&\frac{c_0^ke^{-tc_0}}{k!(k-1)!}\int_{0}^{tc_1}f(t-2 x/c_1) x^{k-1}(t-x/c_1)^k 
\exp\Big\{ \Big( \frac{c_0}{c_1}-1 \Big)x \Big\}dx,
\end{eqnarray*}
where $D_t=\Rset_+^2\cap\{y/c_0+ x/c_1\le t \}$.
Using the change of variable $z=t-2x/c_1$, we obtain $x=c_1\frac{t-z}{2}$, 
$t-x/c_1=\frac{t+z}{2}$ and
\begin{equation}
\label{fin_pair__*}
\Delta_{2k}(f)=\frac{(c_0c_1)^k}{2}\frac{e^{-(c_0+c_1)t/2}}{k!(k-1)!}
\int_{-t}^tf(z)\Big(\ \frac{t-z}{2} \Big)^{k-1}\Big( \frac{t+z}{2} \Big)^k
\exp\{ ( c_1-c_0 )z/2\}dz.
\end{equation}
Finally \eqref{fin_pair__*} and \eqref{constantes_c} imply \eqref{pair__*} and 
\eqref{cond_pair}.\\
{\bf d)} Evaluation of $\Delta_{2k+1}(f)$ for $k\ge 0$.
The arguments are similar to those presented in part c). On the event
$\xi^o_{k+1}/c_0+\xi_k^e/c_1\le t<\xi^o_{k+1}/c_0+\xi_k^e/c_1+ e_{2k+2}/c_1$, 
we have: $\Z_t=2\xi^o_{k+1}/c_0-t$; this implies
\[
\Delta_{2k+1}(f)=\E\Big[ 1_{\{ \xi_{k+1}^o/c_0+\xi_k^e/c_1\le t \}}
\exp\Big( -c_1(t-\xi_{k+1}^o/c_0-\xi_k^e/c_1) \Big) f(2\xi_{k+1}^o/c_0-t) \Big].
\]
Since $\xi_{k+1}^o$ and $\xi_k^e$ are independent and gamma distributed with 
parameter $k+1$ (resp. $k$), we get
\begin{equation}
\label{fin_impair__*}
\Delta_{2k+1}(f)=\frac{c_0^{k+1}c_1^k}{2(k!)^2}e^{-(c_0+c_1)t/2}\int_{-t}^tf(z)
\Big(\ \frac{t-z}{2} \Big)^{k}\Big( \frac{t+z}{2} \Big)^k\exp\Big\{ ( c_1-c_0 )z/2 \Big\}dz.
\end{equation}
This leads directly to \eqref{impair__*} and \eqref{cond_impair}.\hfill{$\Box$}\\[5pt]
\noindent
Let us recall the definition of the modified Bessel functions:
\[
I_\nu(\xi)=\sum_{m\ge0}\frac{(\xi/2)^{\nu+2m}}{m!\Gamma(\nu+m+1)}.
\]
\begin{prop}
\label{bessel__*}
The distribution of $\Z_t$ is given by
\begin{equation}
\label{distrib__*}
\P(\Z_t\in dx)=e^{-c_0t}\delta_t(dx)+e^{-\tau t} f(t,x) 1_{[-t,t]}(x),
\end{equation}
where
\begin{equation}
\label{express__*}
f(t,x)=\frac{1}{2}\Big[ \sqrt{\frac{c_0c_1(t+x)}{t-x}}
I_1\Big( \sqrt{c_0c_1(t^2-x^2)} \Big) +c_0I_0\Big(\sqrt{c_0c_1(t^2-x^2)} \Big)
\Big]e^{\overline{\tau}x}.
\end{equation}
\end{prop}
\begin{rem}
Let us focus our attention to the symmetric case $c_0=c_1$. We can introduce some 
randomization of the initial condition as follows: let $\epsilon$ be a $\{-1,1\}$-valued 
random variable, independent from the Poisson process $N^{c_0}_t$, with
$p:=\P(\epsilon=1)=1-\P(\epsilon=-1)$. It is easy to deduce from \eqref{distrib__*} 
that we have
\begin{equation}
\label{der}
\P(\epsilon Z_t^{c_0}/t\in dx)=\Big(p\delta_1(dx)+(1-p)\delta_{-1}(dx)+g(t,x)dx\Big)e^{-c_0 t},
\end{equation}
with
\[
g(t,x)=\frac{c_0 t}{2}\Big\{I_0\Big(c_0
t\sqrt{1-x^2}\Big)+\frac{1+(2p-1)x}{\sqrt{1-x^2}}I_1\Big(c_0
t\sqrt{1-x^2}\Big)\Big\}1_{[-1,1]}(x)
\]
and $\delta_1(dx)$ (resp. $\delta_{-1}(dx)$) is the Dirac measure at $1$ (resp. $-1$).\\
 In the particular case $p=1/2$, $x\to g(t,x)$ is an even function. G.H. Weiss
(\cite{weiss02} p.393) proved \eqref{der} using an analytic method based on 
Fourier-Laplace transform.
\end{rem}
\noindent
{\it Proof of Proposition \ref{bessel__*}.}
The proof is a direct consequence of the expression of Proposition \ref{loi__1}. 
Indeed, for each bounded continuous function $\varphi$ we denote
\[
\Delta=\E[\varphi(\Z_t)]=\varphi(t)e^{-c_0 t}+\sum_{k\ge 1}\Delta_{2k}(\varphi)
+\sum_{k\ge 0}\Delta_{2k+1}(\varphi)=\varphi(t)e^{-c_0 t}+\Delta_e+\Delta_o,
\]
where $\Delta_n(\varphi)=\E[\varphi(\Z_t)1_{\{ \N_t=n \}}]$. 
Using \eqref{pair__*} and \eqref{cond_pair} we get
\[
\Delta_e=e^{-\tau t}\int_{-t}^t \varphi(z)S_e(z)e^{\overline{\tau}z}dz,
\]
with
\begin{eqnarray*}
S_e(z)&=&\frac{1}{2}\sum_{k\ge 1}\frac{(c_0c_1)^k}{k!(k-1)!}
\Big( \frac{t-z}{2} \Big)^{k-1}\Big( \frac{t+z}{2} \Big)^{k}\\
&=&\frac{1}{2}\sqrt{c_0c_1}\sqrt{\frac{t+z}{t-z}}\sum_{k\ge 0}
\frac{1}{k!(k+1)!}\Big( \frac{\sqrt{c_0c_1(t^2-z^2)}}{2} \Big)^{2k+1}\\
&=&\frac{1}{2}\sqrt{c_0c_1}\sqrt{\frac{t+z}{t-z}}I_1\Big( \sqrt{c_0c_1(t^2-z^2)} \Big).
\end{eqnarray*}
For the odd indexes, by \eqref{impair__*} and \eqref{cond_impair} we get
\[
\Delta_o=e^{-\tau t}\int_{-t}^t \varphi(z)S_o(z)e^{\overline{\tau}z}dz,
\]
with
\[
S_o(z)=\frac{1}{2}\sum_{k\ge 0}\frac{c_0^{k+1}c_1^k}{(k!)^2}\Big( \frac{t^2-z^2}{4} \Big)^{k}
=\frac{c_0}{2}I_0\Big( \sqrt{c_0c_1(t^2-z^2)} \Big).
\]
\hfill{$\Box$}
\begin{prop}
\label{markk}
1) $(\Z_t,\N_t;\ t\ge 0)$ is a $\Rset\times\Nset$-valued Markov process.\\
2) Let $s\ge 0$ and $n\ge 0$. Conditionally on $\Z_s=x$ and
$\N_s=n$, $\Big((\Z_{t+s},\N_{t+s}), \, t\ge 0\Big)$ is distributed as
\begin{eqnarray*}
\left\{\begin{array}{ll}
\Big(\Big( x+\int_0^{t}(-1)^{\N_u}du, n+\N_{t} \Big),\ t\ge 0\Big)&\mbox{when $n$ is even},\\[8pt]
\Big(\Big( x-\int_0^{t}(-1)^{\n_u}du, n+\n_{t} \Big),\ t\ge 0\Big)&\mbox{otherwise}.
\end{array}\right.
\end{eqnarray*}
\end{prop}
\begin{rem}
Note that Propositions \ref{markk} and \ref{loi__1} permit to determine the semigroup of
$\Big((\Z_t,\N_t),\, t\ge 0\Big)$ i.e. $\P(\Z_t\in dx,\, \N_t=n\vert \Z_s=y,\,\N_s=m)$ 
where $t>s$, $n\ge m$ and $y\in[-s,s]$.
\end{rem}
\noindent {\it Proof of Proposition \ref{markk}}.
Let $t>s\ge 0$. Using \eqref{def_1.2} we get
\[
\Z_t=\Z_s+(-1)^{\N_s}\int_0^{t-s}(-1)^{\tilde{N}^s_u}du,
\]
where $\tilde{N}^s_u=\N_{s+u}-\N_s$, $u\ge 0$.\\
Note that $(\tilde{N}_u^{s};\ u\ge 0)\overset{(d)}{=}(\N_u;\ u\ge 0)$ if $\N_s\in 2\Nset$
and $(\tilde{N}_u^{s};\ u\ge 0)\overset{(d)}{=}(\n_u;\ u\ge 0)$ if $\N_s\in 2\Nset+1$.
This shows Proposition \ref{markk}.\hfill{$\Box$}\\[5pt]
Next, we determine (in Proposition \ref{trlp} below) the Laplace transform of the r.v. $\Z_t$. 
It is possible to use the distribution of $\Z_t$ (cf Proposition \ref{bessel__*}), 
but this method has the disadvantage of leading to heavy calculations. We develop 
here a method which uses the fact  that $(\Z_s;\ s\ge 0)$ is a stochastic process 
given by \eqref{def_1.2}. The key tool is Lemma \ref{essai_mark} below. Roughly 
speaking Lemma \ref{essai_mark} gives the generator of the Markov process 
$(\Z_t,\N_t)$. Lemma \ref{essai_mark} is an important ingredient in the proof 
of Proposition  \ref{telegraph**} besides.
\begin{lem}
\label{essai_mark} Let $F:\Rset\times\Nset\to\Rset$ denote a bounded and
continuous function such that $z\to F(z,n)$ is of class
$\mathcal{C}^1$ for all $n$. Then
\begin{eqnarray}
\label{eq:essai_mark} \frac{d}{dt}\E[ F(\Z_t,\N_t) ]&=\E\Big[&
\frac{\partial F}{\partial
z}(\Z_t,\N_t)(-1)^{\N_t}\Big]\nonumber \\
&+\E\Big[&\Big( F(\Z_t,\N_t+1)-F(\Z_t,\N_t) \Big)\nonumber\\
&&\times\Big( c_11_{\{ \N_t\in 2\Nset+1 \}}+c_01_{\{ \N_t\in
2\Nset \}} \Big) \Big].
\end{eqnarray}
\end{lem}
\begin{proof}
Let us denote by $\Delta(t)=\E[F(\Z_t,\N_t)]$. In order to compute
the $t$-derivative we shall decompose the increment of $t\to\Delta(t)$ 
in a sum of two terms:
\[
\frac{\Delta(t+h)-\Delta(t)}{h}=B_h+C_h,
\]
with
\[
B_h=\frac{1}{h}\Big\{
\E[F(\Z_{t+h},\N_{t+h})]-\E[F(\Z_t,\N_{t+h})] \Big\},
\]
\[
C_h=\frac{1}{h}\Big\{ \E[F(\Z_{t},\N_{t+h})]-\E[F(\Z_t,\N_{t})]
\Big\}.
\]
Since $F(\cdot,n)$ is continuously differentiable with respect to the
variable $z$ and $t\to\Z_t$ is differentiable (cf \eqref{def_1.2}), 
using the change of variable formula we obtain
\[
\frac{1}{h}\Big\{ F(\Z_{t+h},\N_{t+h})-F(\Z_t,\N_{t+h})
\Big\}=\frac{1}{h}\int_t^{t+h}\frac{\partial F}{\partial
z}(\Z_u,\N_{t+h})(-1)^{\N_u}du.
\]
Therefore
\begin{equation}
\label{..1} \lim_{h\to 0}B_h=\E\Big[ \frac{\partial F}{\partial
z}(\Z_t,\N_t)(-1)^{\N_t} \Big].
\end{equation}
In order to study the limit of $C_h$, we consider two cases: $\N_t\in 2\Nset$ and $\N_t\in
2\Nset +1$:
\begin{eqnarray*}
C_h&=&\frac{1}{h}
\E\Big[\Big(F(\Z_{t},\N_{t}+\tilde{N}_h^{c_1,c_0})-F(\Z_t,\N_{t})\Big)1_{\{
\N_t\in 2\Nset +1
\}}\Big]\\
&+&\frac{1}{h}
\E\Big[\Big(F(\Z_{t},\N_{t}+\tilde{N}_h^{c_0,c_1})-F(\Z_t,\N_{t})\Big)1_{\{
\N_t\in 2\Nset \}}\Big],
\end{eqnarray*}
where $\tilde{N}_h=\N_{t+h}-\N_t$.\\
According to Proposition \ref{markk}, conditionally on $\Z_t$ and $\N_t\in 2\Nset$ 
(resp. $\N_t\in 2\Nset+1$), $\tilde{N}_h$ is distributed as $\N_h$ (resp. $\n_h$). 
Note that Proposition \ref{loi__1} implies that $\P(\N_h\ge 2)=o(h)$ and
\[
\P(N_h^{c_0,c_1}=1)=\frac{c_0}{2}\Big(\frac{e^{\overline{\tau} h}-e^{-\overline{\tau} h}}{\overline{\tau}}\Big)e^{-\tau h}=c_0 h+o(h).
\]
Consequently
\begin{eqnarray}
\label{..2} \lim_{h\to 0}
C_h&=&c_1\E\Big[\Big(F(\Z_{t},\N_{t}+1)-F(\Z_t,\N_{t})\Big)1_{\{ \N_t\in
2\Nset +1 \}}\Big]\nonumber\\
&+&c_0\E\Big[\Big(F(\Z_{t},\N_{t}+1)-F(\Z_t,\N_{t})\Big)1_{\{ \N_t\in 2\Nset
\}}\Big].
\end{eqnarray}
Then, \eqref{..1} and \eqref{..2} clearly imply
Lemma \ref{essai_mark}.
\end{proof}
Let us introduce the two quantities:
\begin{equation}
\label{lapp}
L_e(t)=\E\Big[ e^{-\mu \Z_t}1_{\{ \N_t\in 2\Nset \} }
\Big]\ \mbox{and}\ L_o(t)=\E\Big[ e^{-\mu \Z_t}1_{\{ \N_t\in
2\Nset+1 \} } \Big],\ (t\ge 0,\mu\in\Rset).
\end{equation}
Since $\vert \Z_t\vert\le t$, then $L_e(t)$ and $L_o(t)$ are well defined for 
any $\mu\in\Rset$. Note that $\mu\to L_e(t)$ (resp. $\mu\to L_o(t)$) is a Laplace 
transform. We have mentioned the $t$-dependency only because it will play an 
important role in our proof of Proposition \ref{trlp} below.
\begin{prop}
\label{trlp} Let $L_e(t)$ and $L_o(t)$ be defined by \eqref{lapp}. Then
\begin{equation}
\label{nn*} L_e(t)=\frac{1}{\sqrt{\mathcal{E}}}\left(
(-\mu+\overline{\tau})\sinh(t\sqrt{\mathcal{E}})+\sqrt{\mathcal{E}}\cosh(t\sqrt{\mathcal{E}})
\right)e^{-\tau t},
\end{equation}
\begin{equation}
\label{nn**}
L_o(t)=\frac{c_0}{\sqrt{\mathcal{E}}}\sinh(t\sqrt{\mathcal{E}})e^{-\tau
t},
\end{equation}
\begin{equation}
\label{resum**} \E[e^{-\mu\Z_t}]=\frac{1}{\sqrt{\mathcal{E}}}\Big[
(-\mu+\tau)\sinh(t\sqrt{\mathcal{E}})+\sqrt{\mathcal{E}}\cosh(t\sqrt{\mathcal{E}})
\Big]e^{-\tau t},
\end{equation}
where $\mathcal{E}=\mu^2-2\overline{\tau}\mu+\tau^2$.
\end{prop}
\begin{proof}
Applying Lemma \ref{essai_mark} with the particular function
$F(z,n)=e^{-\mu z}1_{\{n\in 2\Nset \}}$, we have:
\begin{eqnarray*}
\frac{d}{dt}L_e(t)&=&-\mu\E\Big[
e^{-\mu\Z_t}(-1)^{\N_t}1_{\{\N_t\in
2\Nset\}} \Big]\\
&+&E\Big[ e^{-\mu\Z_t}\Big( 1_{\{ \N_t\in 2\Nset+1 \}}-1_{\{
\N_t\in 2\Nset \}} \Big)\\&&\times \Big(c_1 1_{\{ \N_t\in 2\Nset+1
\}}+c_01_{\{ \N_t\in 2\Nset \}} \Big)\Big]
\end{eqnarray*}
We deduce
\[
\frac{d}{dt}L_e(t)=-(\mu+c_0) L_e(t)+c_1L_o(t).
\]
Similarly
\begin{eqnarray*}
\frac{d}{dt}L_o(t)&=&-\mu\E\Big[
e^{-\mu\Z_t}(-1)^{\N_t}1_{\{\N_t\in
2\Nset+1\}} \Big]\\
&+&E\Big[ e^{-\mu\Z_t}\Big( 1_{\{ \N_t\in 2\Nset \}}-1_{\{ \N_t\in
2\Nset+1 \}} \Big)\\&&\times \Big(c_1 1_{\{ \N_t\in 2\Nset+1
\}}+c_01_{\{ \N_t\in 2\Nset \}} \Big)\Big].
\end{eqnarray*}
We get therefore
\[
\frac{d}{dt}L_o(t)=(\mu-c_1) L_o(t)+c_0L_e(t).
\]
 To sum up
 \begin{eqnarray*}
 \frac{d}{dt}\left(\begin{array}{l}L_e(t)\\
 L_o(t)\end{array}\right)=\left(\begin{array}{cc}-\mu-c_0 & c_1\\
 c_0 & \mu-c_1\end{array}\right)\left(\begin{array}{l}L_e(t)\\
 L_o(t)\end{array}\right).
 \end{eqnarray*}
 We deduce the expressions of $L_e(t)$ and
 $L_o(t)$:
 \begin{equation}
 \label{eta1}
 L_e(t)=a_+e^{\lambda_+ t}+a_-e^{\lambda_-t}\quad\quad L_o(t)=b_+e^{\lambda_+
 t}+b_-e^{\lambda_-t},
 \end{equation}
 where $\lambda_\pm=-\tau\pm\sqrt{\mu^2-2\overline{\tau}\mu+\tau^2}=
 -\tau\pm\sqrt{\mathcal{E}}$.\\
 The constants $a_\pm$ and $b_\pm$ are evaluated with the initial conditions:
 \[
 L_e(0)=\P(\N_0\in 2\Nset)=1,\quad\quad L_o(0)=\P(\N_0\in
 2\Nset+1)=0,
 \]
 \[
 \frac{dL_e}{dt}(0)=-(\mu+c_0)L_e(0)+c_1L_o(0)=-\mu-c_0,
 \]
 \[
 \frac{dL_o}{dt}(0)=(\mu-c_1)L_o(0)+c_0L_e(0)=c_0.
 \]
 We obtain
 \begin{equation}
 \label{eta2}
 a_+=\frac{1}{2\sqrt{\mathcal{E}}}(-\mu+\overline{\tau}+\sqrt{\mathcal{E}})\quad\mbox{and}\quad
 a_-=\frac{1}{2\sqrt{\mathcal{E}}}(\mu-\overline{\tau}+\sqrt{\mathcal{E}}),
 \end{equation}
 \begin{equation}
 \label{eta3}
 b_+=\frac{c_0}{2\sqrt{\mathcal{E}}}\quad\mbox{and}\quad b_-=-\frac{c_0}{2\sqrt{\mathcal{E}}}
 \end{equation}
 Using \eqref{eta1}, \eqref{eta2} and \eqref{eta3}, Proposition \ref{trlp} follows.
\end{proof}
It is easy to deduce two direct consequences of Proposition \ref{trlp}. 
First, taking $\mu=0$ we obtain $\P(\N_t\in 2\Nset)$ and $\P(\N_t\in 2\Nset+1)$. 
Second, taking the expectation in \eqref{def_1.2} we get the mean of $\Z_t$.
\begin{cor} We have:
\[
\P(\N_t\in 2\Nset)=\frac{1}{\tau}\Big[ \overline{\tau}\sinh(\tau
t)+\tau\cosh(\tau t) \Big]e^{-\tau t},
\]
\[
\P(\N_t\in 2\Nset+1)=\frac{c_0}{\tau} \sinh(\tau t)e^{-\tau t},
\]
and
\[
\E[\Z_t]=\frac{\overline{\tau}}{\tau}t+\frac{c_0}{2\tau^2}(1-e^{-2\tau
t}).
\]
\end{cor}
\begin{rem}\label{remark1} The Laplace transform with respect to the time 
variable can also be explicitly computed. We define $F(\mu,s)=\int_0^\infty
e^{-st}\E[e^{-\mu \Z_t}]dt$. Integrating \eqref{resum**} with respect to $dt$ 
we get
\begin{eqnarray*}
F(\mu,s)&=&\frac{1}{2\sqrt{\mathcal{E}}}\Big(\sqrt{\mathcal{E}}+(-\mu+\tau)
\Big)\frac{1}{s-\sqrt{\mathcal{E}}+\tau}+\frac{1}{2\sqrt{\mathcal{E}}}
\Big(\sqrt{\mathcal{E}}-(-\mu+\tau)\Big)\frac{1}{s+\sqrt{\mathcal{E}}+\tau}\\
&=&\frac{(\sqrt{\mathcal{E}}-\mu+\tau)(s+\sqrt{\mathcal{E}}+\tau)+
(\sqrt{\mathcal{E}}+\mu-\tau)(s-\sqrt{\mathcal{E}}+\tau)}{2\sqrt{\mathcal{E}}
((s+\tau)^2-\mathcal{E})}\\
&=&\frac{2s\sqrt{\mathcal{E}}+4\tau\sqrt{\mathcal{E}}-2\mu
\sqrt{\mathcal{E}}}{2\sqrt{\mathcal{E}}((s+\tau)^2-\mathcal{E})}=
\frac{ s+ 2\tau-\mu}{(s+\tau)^2-\mathcal{E}}
\end{eqnarray*}
In the symmetric case, $\mathcal{E}$ equals $\mu^2+c_0^2$, then
\begin{equation}
\label{...}
F(\mu,s)=\frac{s+ 2c_0-\mu }{s^2+2sc_0-\mu^2}.
\end{equation}
Let $(Z_t)$ be the symmetrization of $(Z^{c_0}_t)$ which is defined by an 
initial randomization:
\[
Z_t=\epsilon Z_t^{c_0},\quad t\ge 0,
\]
where $\epsilon$ is independent of $Z_t^{c_0}$ and $\P(\epsilon=\pm 1)=1/2$.\\
Relation \eqref{...} implies
\[
\int_0^\infty e^{-st}\E[e^{-\mu Z_t}]dt=\frac{s+ 2c_0}{s^2+2sc_0-\mu^2}.
\]
This identity has been obtained by Weiss in \cite{weiss02}.
\end{rem}
Let us now present a link between the ITN process and the
telegraph
equation in the particular symmetric case $c_0=c_1=c>0$. Recall that $(N_t^c)$ 
is a Poisson process with parameter $c$.\\
Let $f:\Rset\to\Rset$ be a function of class $\mathcal{C}^2$ whose
first and second derivatives are bounded. We define
\[
u(x,t)=\frac{1}{2}\Big\{f(x+at)+f(x-at)\Big\},\quad x\in\Rset,\ t\ge 0.
\]
Then (cf \cite{eckstein00}) $u$ is the unique solution of the wave equation
\begin{eqnarray*}
\left\{\begin{array}{l}
\displaystyle\frac{\partial^2 u}{\partial t^2}=a^2\frac{\partial^2 u}{\partial x^2},\\
\displaystyle u(x,0)=f(x),\quad \frac{\partial u}{\partial
t}(x,0)=0.
       \end{array}\right.
\end{eqnarray*}
\begin{prop}
\label{telegraph**} The function
\[
w(x,t)=\E\Big[ u\Big(x,\int_0^t (-1)^{N^c_s}ds\Big)\Big],\quad
(x\in\Rset, t\ge 0)
\]
 is the
solution of the telegraph equation (TE)
\begin{eqnarray*}
\left\{\begin{array}{l}
\displaystyle\frac{\partial^2 w}{\partial t^2}+2c\frac{\partial w}{\partial t}=
a^2\frac{\partial^2 w}{\partial x^2},\\[8pt]
\displaystyle w(x,0)=f(x),\quad \frac{\partial w}{\partial
t}(x,0)=0.
       \end{array}\right.
\end{eqnarray*}
\end{prop}
This result can be proved using asymptotic analysis applied to
difference equation associated with the persistent random walk
\cite{goldstein50} or using Fourier transforms \cite{weiss02}.
Here we shall present a new proof.\\
{\it Proof of Proposition \ref{telegraph**}.} 
Applying twice Lemma \ref{essai_mark} to $(z,n)\to u(x,z)$ and 
$(z,n)\to \frac{\partial u}{\partial t}(x,z)(-1)^n$ we obtain:
\[
\frac{\partial w}{\partial t}(x,t)=\E\Big[ \frac{\partial u}{\partial t}
\Big(x,\int_0^t (-1)^{N^c_s}ds\Big) (-1)^{N^c_t} \Big].
\]
and
\[
\frac{\partial^2 w}{\partial t^2}(x,t)=\E\Big[  \frac{\partial^2 u}{\partial t^2}
\Big(x,\int_0^t (-1)^{N^c_s}ds\Big)\Big]-2c\E\Big[  \frac{\partial u}{\partial t}
\Big(x,\int_0^t (-1)^{N^c_s}ds\Big) (-1)^{N^c_t} \Big].
\]
Since $u$ solves the wave equation we have
\[
\frac{\partial^2 w}{\partial t^2}(x,t)=a^2\frac{\partial^2 w}{\partial x^2}(x,t)
-2c\frac{\partial w}{\partial t}(x,t).
\]
The function $w$ is actually the solution of the telegraph
equation. It is easy to prove that $w$ satisfies the boundary
conditions.\hfill{$\Box$}\\[5pt]

Let us note that Proposition \ref{telegraph**} can be extended to the asymmetric 
case $c_0\neq c_1$. In this general case the telegraph equation is replaced by a 
linear system of partial differential equations.
\begin{rem}
\label{telegraph_gen} 
1) In \cite{eckstein00}, \cite{griego71}, an extension of Proposition \ref{telegraph**} 
has been proved. Let $A$ be the generator of a strongly continuous group of 
bounded linear operators on a Banach space. If $w$ is the unique solution of 
this abstract "wave equation":
\[
\frac{\partial^2 w}{\partial t^2}=A^2 w;\ w(\cdot,0)=f,\ 
\frac{\partial w}{\partial t}(\cdot,0)=Ag\quad (f,g\in\mathcal{D}(A))
\]
then $u(x,t)=\E\Big[ w\Big(x,\int_0^t (-1)^{N^c_s}ds\Big) \Big]$ 
solves the abstract "telegraph equation":
\[
\frac{\partial^2 u}{\partial t^2}=A^2 u-2c\frac{\partial u}{\partial t},\ 
\ u(\cdot, 0)=f,\ \ \frac{\partial u}{\partial t}(\cdot, 0)=Ag.
\]
2) In the same vein as \cite{griego71}, Enriquez \cite{enriquez07} has introduced processes with jumps to represent solutions of some linear differential equations and biharmonic equations in the presence of a potential term. Moreover useful references are given in \cite{enriquez07}.\\
3) It is easy to deduce from Lemma \ref{essai_mark} that the functions
\[
w_e(x,t)=\E\Big[ u\Big(x,\int_0^t (-1)^{N^{c_0,c_1}_s}ds\Big)
1_{\{N^{c_0,c_1}_t\in 2\Nset \}}\Big],\quad
(x\in\Rset, t\ge 0)
\]
\[
w_o(x,t)=\E\Big[ u\Big(x,\int_0^t (-1)^{N^{c_0,c_1}_s}ds\Big)
1_{\{N^{c_0,c_1}_t\in 2\Nset+1 \}}\Big],\quad
(x\in\Rset, t\ge 0)
\]
are
solutions of the general telegraph system (TS)
\begin{eqnarray*}
\left\{\begin{array}{l}
\displaystyle\frac{\partial^2 w_e}{\partial t^2}
=(c_0c_1-c_0^2)w_e+(c_0c_1-c_1^2)w_o-2c_0\frac{\partial w_e}{\partial t}
+a^2\frac{\partial^2 w_e}{\partial x^2},\quad \\[8pt]
\displaystyle\frac{\partial^2 w_o}{\partial t^2}
=(c_0c_1-c_0^2)w_e+(c_0c_1-c_1^2)w_o-2c_1\frac{\partial w_o}{\partial t}
+a^2\frac{\partial^2 w_o}{\partial x^2},\quad \\[8pt]

\displaystyle w_e(x,0)=f(x),\quad w_o(x,0)=0\quad \frac{\partial w_e}{\partial
t}(x,0)=-c_0 f(x)\quad \frac{\partial w_o}{\partial
t}(x,0)=c_0f(x).
       \end{array}\right.
\end{eqnarray*}
\end{rem}
{\color{blue}
\ifthenelse{\boolean{complet}}{
\begin{proof}Let us consider the increment: $I_e(h)=w_e(x,t+h)-w_e(x,t)$. Then
\[
I_e(h)=\E\Big[ u\Big(x,\int_0^{t+h}(-1)^{N_s}ds\Big)1_{\{N_{t+h}\in 2\Nset \}}
- u\Big(x,\int_0^{t}(-1)^{N_s}ds\Big)1_{\{N_{t}\in 2\Nset \}}\Big]
\]
\begin{eqnarray*}
I_e(h)&=&\E\Big[ \Big\{ u\Big(x,\int_0^{t}(-1)^{N_s}ds\Big)+\int_t^{t+h}
\frac{\partial u}{\partial t}\Big( x,\int_0^s (-1)^{N_u}du \Big)(-1)^{N_s}ds \Big\} 
1_{\{ N_{t+h}\in 2\Nset \}}\Big]\\
&-&E\Big[u\Big(x,\int_0^{t}(-1)^{N_s}ds\Big)1_{\{N_{t}\in 2\Nset \}}\Big]\\
&=&A_h+B_h+C_h,
\end{eqnarray*}
where
\[
A_h=\E\Big[ u\Big(x,\int_0^{t}(-1)^{N_s}ds\Big)\Big(1_{\{ N_{t+h}\in 2\Nset \}
\cap \{N_t\in 2\Nset\}}-1_{\{ N_t\in 2\Nset \}}\Big)\Big]
\]
\[
B_h=\E\Big[ u\Big(x,\int_0^{t}(-1)^{N_s}ds\Big)1_{\{ N_{t+h}\in 2\Nset \}
\cap \{N_t\in 2\Nset+1\}}\Big]
\]
\[
C_h=\E\Big[ \int_t^{t+h}\frac{\partial u}{\partial t}
\Big( x,\int_0^s (-1)^{N_u}du \Big)(-1)^{N_s}ds 1_{\{ N_{t+h}\in 2\Nset \}} \Big].
\]
Using the right-continuity of $N_t$ we deduce:
\[
\lim_{h\to 0}\frac{1}{h}C_h=\E\Big[ \frac{\partial u}{\partial t}
\Big( x , \int_0^t (-1)^{N_s}ds\Big) 1_{\{ N_t\in 2\Nset \}} \Big].
\]
\begin{eqnarray*}
A_h&=&\E\Big[ u\Big(x,\int_0^{t}(-1)^{N_s}ds\Big)1_{\{N_t\in 2\Nset\}}\Big]
\Big(\P(N_{t+h}\in 2\Nset\vert N_t\in 2\Nset ) -1\Big)\\
&=&w_e(x,t)(-hc_0+o(h)).
\end{eqnarray*}
By similar arguments, we get
\[
B_h=w_o(x,t)\P(N_{t+h}\in 2\Nset \vert N_t\in 2\Nset+1)=w_o(x,t)(hc_1+o(h))
\]
We deduce
\begin{equation}
\label{gen_deriv}
\frac{\partial w_e}{\partial t}(x,t)=-c_0w_e(x,t)+c_1w_o(x,t)+\tilde{w}_e(x,t),
\end{equation}
where $\tilde{w}_e$ is defined by
\[
\tilde{w}_e(x,t)=\E\Big[ \frac{\partial u}{\partial t}\Big( x , 
\int_0^t (-1)^{N_s}ds\Big) 1_{\{ N_t\in 2\Nset \}} \Big].
\]
Similar computation leads to
\begin{equation}
\label{gen_deriv_1}
\frac{\partial w_o}{\partial t}(x,t)=-c_1w_o(x,t)+c_0w_e(x,t)-\tilde{w}_o(x,t),
\end{equation}
where $\tilde{w}_o$ is defined by
\[
\tilde{w}_o(x,t)=\E\Big[ \frac{\partial u}{\partial t}\Big( x , 
\int_0^t (-1)^{N_s}ds\Big) 1_{\{ N_t\in 2\Nset +1\}} \Big].
\]
We apply the same procedure to $\tilde{w}$:
\begin{equation}
\label{gen_deriv_2}
\frac{\partial \tilde{w}_e}{\partial t}(x,t)=-c_0\tilde{w}_e(x,t)+c_1\tilde{w}_o(x,t)
+\E\Big[ \frac{\partial ^2u}{\partial t^2}\Big( x , \int_0^t (-1)^{N_s}ds\Big) 
1_{\{ N_t\in 2\Nset \}} \Big].
\end{equation}
By \eqref{gen_deriv}, \eqref{gen_deriv_1} and since $u$ is the solution of 
the wave equation, we obtain
\begin{eqnarray*}
\frac{\partial^2 w_e}{\partial t^2}(x,t)&=&-c_0\frac{\partial w_e}{\partial t}(x,t)
+c_1\frac{\partial w_o}{\partial t}(x,t)+\frac{\partial \tilde{w}_e}{\partial t}(x,t)\\
&=&-c_0\frac{\partial w_e}{\partial t}(x,t)+c_1\frac{\partial w_o}{\partial t}(x,t)
-c_0\tilde{w}_e(x,t)+c_1\tilde{w}_0(x,t)+a^2\frac{\partial^2 w_e}{\partial x^2}(x,t)\\
&=&-2c_0\frac{\partial w_e}{\partial t}(x,t)+(c_0 c_1-c_0^2)w_e(x,t)+(c_0 c_1
-c_1^2)w_o(x,t)+a^2\frac{\partial^2 w_e}{\partial x^2}(x,t).
\end{eqnarray*}
The same kind of computation permits to obtain the equation satisfied by $w_o$.
\end{proof}\newpage
}{}}
{\color{blue}
\ifthenelse{\boolean{complet}}{
\section{Properties of the integrated telegraph noise}\label{properties}
The study of the process $(Z^0_t,\ t\ge 0)$ introduced in
\eqref{def_1.2} shall be only given in the particular symmetric case $c_0=c_1$ 
(the study in the general case is complicated). Consequently, the r.v.'s 
$(e_n;\ n\ge 1)$ defined at the beginning of Section \ref{results}, have 
the common exponential distribution with parameter $1/c_0$. Hence the 
process $(N^0_t,\ t\ge0)$ defined by \eqref{...} is a Poisson process with 
rate $\tau=c_0$
($\E(N^0_t)=\tau t$) and $N_0=0$. Recall that
\begin{equation}\label{defx}
Z^0_t=\int_0^t(-1)^{N^0_s}ds.
\end{equation}
The aim of this section is to study the two dimensional process $(Z_t^0,\
N_t^0;\ t\ge 0)$.
First, we determine in Proposition \ref{loi cond} the conditional density of 
$Z_t^0$ given $N_t^0=n$. As a by product we obtain the density function of 
$Z_t^0$ (see Proposition \ref{bessel} with $p=1$). Second, we prove in 
Proposition \ref{ajout} that $(Z_t^0,N_t^0,\, t\ge 0)$ is Markov and we 
determine its semi-group. We conclude this section by showing that the 
solution of the telegraph equation can be expressed in terms of the associated 
wave equation and $(Z_t^0)_{t\ge 0}$. For this reason, $(Z_t^0)_{t\ge 0}$ will 
be called the integrated telegraph noise (ITN for short).
\begin{prop}\label{loi cond}
Let $n\in\mathbb{N}$ and $t>0$. The conditional density function
of $Z^0_t$ given $N^0_t=n$ is the centered beta density
\begin{equation}\label{eq:loi cond1}
\mbox{for}\ n=2k,\ k\in\mathbb{N^*}:\quad
f_n(t,x)=\chi_{2k}\frac{(t+x)^k(t-x)^{k-1}}{t^{2k}}1_{[-t\le x\le t
]}(x),
\end{equation}
\begin{equation}\label{eq:loi cond2}
\mbox{for}\ n=2k+1,\ k\in\mathbb{N}:\quad
f_n(t,x)=\chi_{2k+1}\frac{(t+x)^k(t-x)^{k}}{t^{2k+1}}1_{[-t\le x\le t
]}(x),
\end{equation}
with
\begin{eqnarray*}
\chi_{2k+1}=\chi_{2k+2}=\frac{1}{2^{2k+1}B(k+1,k+1)}=\frac{(2k+1)!}{2^{2k+1}(k!)^2}\quad
(k\ge 0),
\end{eqnarray*}
where $B$ is the beta function (first Euler function):
$B(r,s)=\frac{\Gamma(r)\Gamma(s)}{\Gamma(r+s)}$.
\end{prop}
\begin{proof} Straightforward for $n=0$.\\
Let us study the two first cases: $n=1$ and $n=2$.\\
{\bf 1)} For $n=1$, we define $U_1$ the first and unique jump for
the counting process $(N^0_t)$ on the time interval $[0,t]$. Then
$Z^0_t=U_1-(t-U_1)=2U_1-t$. Given $N^0_t=1$, $U_1$ is uniformly
distributed on the interval $[0,t]$. This implies that $Z_t^0$ is
uniformly distributed on $[-t,t]$. This corresponds to
\eqref{eq:loi cond2}
with $k=0$ and $c_1=1/2$.\\
{\bf 2)} If $n=2$, we denote by $U_1$ and $U_2$ the jumps of the
counting process with $U_1<U_2$. Given $N^0_t=2$, the couple
$(U_1, U_2)$ is distributed as a couple of rearranged independent
uniform random variables on $[0,t]$. Namely, its density function
is
\[
\frac{2}{t^2}1_{\{0<x<y<t\}}.
\]
Moreover, it is clear that:
\[
Z^0_t=U_1-(U_2-U_1)+t-U_2=2U_1-2U_2+t.
\]
Hence, for any bounded Borel function $\varphi$
\[
\E[\varphi(Z^0_t)\vert
N^0_t=2]=\frac{2!}{t^2}\int_{[0,t]^2}\varphi(2x-2y+t)1_{\{0<x<y<t\}}dx\,dy.
\]
Let us introduce the change of variables $r=2x-2y+t$, $v=x$. Then
\[
\E[\varphi(Z^0_t)\vert
N^0_t=2]=\frac{2!}{2t^2}\int_{-t}^t\varphi(r)\Big(\int_0^{(t+r)/2}dv\Big)\,dr=
\int_{-t}^t\varphi(r)\frac{t+r}{2t^2}dr.
\]
We conclude that the conditional density satisfies \eqref{eq:loi
cond1} with parameters $k=1$ and $c_2=1/2$. 
{\bf 3)} In order to prove the statement, we shall proceed by induction.\\
{\bf a)} Let us begin with few reminders concerning the Dirichlet law.\\
Let $V_1,\ldots, V_n$ be independent and uniformly distributed random variables on $[0,t]$.
We denote
$V_{(1)},\ldots,V_{(n)}$ its increasing rearrangement :
$\{V_{(1)},\ldots,V_{(n)}\}=\{V_1,\ldots, V_n\}$ and
$V_{(1)}<\ldots<V_{(n)}$. Then the Dirichlet law $\mathbb{D}_n(t)$ is the 
distribution of the random vector $(V_{(1)},\ldots,V_{(n)})$: the density 
of $(V_{(1)},\ldots,V_{(n)})$ is given by
\begin{equation}
\label{densdiri} \frac{n!}{t^n}1_{\{0<x_1<\ldots<x_n<t\}}.
\end{equation}
{\bf b)} Let us denote by $(U_n)_{n\ge 1}$ the sequence of jump
times of $(N^0_u)$
\[
N^0_u=\sum_{n\ge 1}1_{\{U_n\le u\}}.
\]
Note that, according to \eqref{...}, $U_n=e_1+\ldots+e_n$.
We recall that, given $N^0_t=n\ge 1$, the distribution of $(U_1,\ldots, U_n)$ 
is $\mathbb{D}_n(t)$.\\
{\bf c)} Let $k\ge 1$. Since
\begin{eqnarray*}
N^0_t=\left\{\begin{array}{l}
2i,\quad \mbox{if}\ t\in[U_{2i}, U_{2i+1}[\quad (i\ge 0)\\
2i+1,\quad \mbox{if}\ t\in[U_{2i+1}, U_{2i+2}[
           \end{array}\right.
\end{eqnarray*}
we deduce that, under the condition $N^0_t=2k$ then
$U_{2k}\le t<U_{2k+1}$ and
\begin{eqnarray}
\label{decomp+-}
Z^0_t&=& U_1-(U_2-U_1)+(U_3-U_2)+\ldots-(U_{2k}-U_{2k-1})+t-U_{2k}\nonumber\\
&=& 2U_1-2U_2+\ldots+2U_{2k-1}-2U_{2k}+t.
\end{eqnarray}
Note that \eqref{decomp+-} still holds if $k=0$ as soon as we adopt the convention $U_0=0$.\\
Similarly, if $N^0_t=2k+1$, we get $U_{2k+1}\le t<U_{2k+2}$
and
\begin{equation}
\label{decompimp} Z^0_t=2U_1-2U_2+\ldots-2U_{2k}+2U_{2k+1}-t.
\end{equation}
We introduce $(R_n)_{n\ge 1}$ the following sequence of random variables
\begin{equation}
\label{suitern} R_n=2\Big(\sum_{i=1}^n (-1)^{i+1} U_i \Big)
\end{equation}
Hence
\begin{eqnarray}
\label{4E} Z^0_t=\left\{\begin{array}{ll}
R_{2k}+t&\mbox{as}\ N^0_t=2k;\ k\ge 0\\
R_{2k+1}-t&\mbox{as}\ N^0_t=2k+1,\ k\ge 0.
           \end{array}\right.
\end{eqnarray}
{\bf d)} Let $\varphi: \Rset\to\Rset_+$ a bounded Borel function and 
$I_n(\varphi,t)$ the associated integral:
\[
I_n(\varphi,t):=\E[\varphi(Z^0_t)\vert N^0_t=n].
\]
Let $k\ge 0$. By \eqref{4E}, \eqref{suitern} and \eqref{densdiri}, we obtain
\[
I_{2k+1}(\varphi,t)=\frac{(2k+1)!}{t^{2k+1}}\int_{[0,t]^{2k+1}}
\varphi(2x_1-2x_2+...+2x_{2k+1}-t)1_{\{x_1<...<x_{2k+1}\}}dx_1...
dx_{2k+1}
\]
\[
I_{2k}(\varphi,t)=\frac{(2k)!}{t^{2k}}\int_{[0,t]^{2k}}\varphi(2x_1-2x_2+
\ldots-2x_{2k}+t)1_{\{x_1<\ldots<x_{2k}\}}dx_1\ldots
dx_{2k}.
\]
Setting $\varphi^{(a)}(x)=\varphi(x+a)$, we get:
\begin{eqnarray}
\label{4F}
I_{2k}(\varphi,t)&=&\frac{(2k)!}{t^{2k}}\int_0^t dx_{2k}
\int_{[0,x_{2k}]^{2k-1} }\varphi(2x_1-2x_2+\ldots+2x_{2k-1}+t-2x_{2k})\nonumber\\
&&1_{\{x_1<\ldots<x_{2k-1}\}}dx_1\ldots dx_{2k-1}\quad (k \ge 1)\nonumber\\
&=&\frac{2k}{t^{2k}}\int_0^t (x_{2k})^{2k-1}
I_{2k-1}(\varphi^{(t-x_{2k})},x_{2k}) dx_{2k}.
\end{eqnarray}
Similarly it can be easily proved:
\begin{equation}
\label{4G} I_{2k+1}(\varphi,t)=\frac{2k+1}{t^{2k+1}}\int_0^t
(x_{2k+1})^{2k} I_{2k}(\varphi^{(x_{2k+1}-t)},x_{2k+1}) dx_{2k+1}.
\end{equation}
{\bf e)} We claim that
\begin{equation}
\label{4H}
I_{2k+1}(\varphi,t)=\alpha_k\int_{-t}^t\frac{(t+x)^k(t-x)^k}{t^{2k+1}}\,\varphi(x)
dx\quad k\ge 0
\end{equation}
\begin{equation}
\label{4I}
I_{2k}(\varphi,t)=\beta_k\int_{-t}^t\frac{(t+x)^k(t-x)^{k-1}}{t^{2k}}\,\varphi(x)
dx\quad k\ge 1.
\end{equation}
Let us observe that $I_0(\varphi,t)=\varphi(t)$. Then, taking
$k=0$ in \eqref{4G} leads to:
\begin{eqnarray*}
I_1(\varphi,t)&=&\frac{1}{t}\int_0^t I_0(\varphi^{(x-t)},x)dx=\frac{1}{t}\int_0^t \varphi^{(x-t)}(x)dx\\
&=&\frac{1}{t}\int_0^t
\varphi(2x-t)dx=\frac{1}{2t}\int_{-t}^t\varphi(y)dy.
\end{eqnarray*}
As a result, \eqref{4H} holds with $k=0$ and $\alpha_0=1/2$.\\
Next, taking $k=1$ in \eqref{4F}, we obtain:
\begin{eqnarray*}
I_2(\varphi,t)&=&\frac{2}{t^2}\int_0^t x I_1(\varphi^{(t-x)},x)dx=
\frac{2}{t^2}\int_0^t\frac{x}{2x}
\Big(\int_{-x}^x\varphi(t-x+y)dy\Big)dx
\end{eqnarray*}
Introducing the change of variables $u=t-x+y$, we have successively:
\begin{eqnarray*}
I_2(\varphi,t)
&=&\frac{1}{t^2}\int_0^t\Big(\int 1_{\{t-2x\le u\le t\}}\varphi(u)du\Big)dx\\
&=&\frac{1}{t^2}\int_{-t}^t\varphi(u)\Big(\int_{(t-u)/2}^t
dx\Big)du=\frac{1}{2 t^2}\int_{-t}^t\varphi(u)(t+u)du.
\end{eqnarray*}
This shows that \eqref{4I} is satisfied with $k=1$ and $\beta_1=1/2.$\\
We prove \eqref{4H} and \eqref{4I} by induction on $k$.\\
First changing $k$ by $k+1$ in \eqref{4F}, we get:
\begin{eqnarray*}
I_{2k+2}(\varphi,t)&=&\frac{2k+2}{t^{2k+2}}\int_0^t x^{2k+1} I_{2k+1}(\varphi^{(t-x)},x) dx\\
&=&\frac{(2k+2)\alpha_k}{t^{2k+2}}\int_{0}^t x^{2k+1}\frac{1}{x^{2k+1}}
\Big(\int_{-x}^x (y+x)^k(x-y)^k\varphi(t-x+y)dy\Big)dx
\end{eqnarray*}
Second introducing
$u=t-x+y$ we have
\begin{eqnarray*}
I_{2k+2}(\varphi,t)
&=&\frac{(2k+2)\alpha_k}{t^{2k+2}}\int_{-t}^t\varphi(u)
\Big(\int_{(t-u)/2}^t(u-t+2x)^k(t-u)^k dx \Big)du\\
&=&\frac{\alpha_k}{t^{2k+2}}\int_{-t}^t\varphi(u)(t-u)^k
(u+t)^{k+1}du.
\end{eqnarray*}
Hence \eqref{4I} (with $k\to k+1$) is satisfied and
\begin{equation}
\label{4K} \beta_{k+1}=\alpha_k.
\end{equation}
We proceed similarly, dealing with \eqref{4G} (with the change $k\to k+1$):
\begin{eqnarray*}
I_{2k+3}(\varphi,t)&=&\frac{2k+3}{t^{2k+3}}\int_0^t x^{2k+2} I_{2k+2}(\varphi^{(x-t)},x) dx\\
&=&\frac{(2k+3)\beta_{k+1}}{t^{2k+3}}\int_{0}^t x^{2k+2}\frac{1}{x^{2k+2}}
\Big(\int_{-x}^x (y+x)^{k+1}(x-y)^k\varphi(x-t+y)dy\Big)dx
\end{eqnarray*}
The change of variable $u=x-t+y$ leads to
\begin{eqnarray*}
I_{2k+3}(\varphi,t)
&=&\frac{(2k+3)\beta_{k+1}}{t^{2k+3}}\int_{-t}^t\varphi(u)
\Big(\int_{(t+u)/2}^t(2x-u-t)^k(t+u)^{k+1} dx \Big)du\\
&=&\frac{(2k+3)\beta_{k+1}}{2(k+1)t^{2k+3}}\int_{-t}^t\varphi(u)(t-u)^{k+1}
(u+t)^{k+1}du.
\end{eqnarray*}
Hence \eqref{4H} follows with $ k+1$ instead of $k$ and
\begin{equation}
\label{4L} \alpha_{k+1}=\frac{2k+3}{2(k+1)}\,\beta_{k+1}.
\end{equation}
It remains to compute the coefficients $(\alpha_k)$ and $(\beta_k)$. 
Since $\alpha_0=1/2$, relations
\eqref{4K} and \eqref{4L} imply that:
\[
\alpha_0=\frac{1}{2}\quad\quad\alpha_{k+1}=\frac{2k+3}{2(k+1)}\;\alpha_k.
\]
Consequently
\begin{eqnarray*}
\alpha_k&=&\frac{(2k+1)(2k-1)\ldots 3\cdot 1}{2^{k+1}k!}=\frac{1}{2^{2k+1}}\frac{(2k+1)!}{(k!)^2}=
\frac{1}{2^{2k+1}}\frac{\Gamma(2k+2)}{\Gamma (k+1)^2}\\
&=&\frac{1}{2^{2k+1}}\frac{1}{B(k+1,k+1)}.
\end{eqnarray*}\end{proof}
Let us recall the definition of the modified Bessel functions:
\[
I_\nu(\xi)=\sum_{m\ge0}\frac{(\xi/2)^{\nu+2m}}{m!\Gamma(\nu+m+1)}.
\]
\begin{prop}
\label{bessel} Let $\epsilon$ be a random variable independent
from the Poisson process $N^0_t$, valued in $\{-1,1\}$ with
$p:=\P(\epsilon=1)=1-\P(\epsilon=-1)$. Then
\[
\P(\epsilon Z_t^0/t\in dx)=\Big(p\delta_1(dx)+(1-p)\delta_{-1}(dx)+f(t,x)dx\Big)e^{-\tau t},
\]
with
\[
f(t,x)=\frac{\tau t}{2}\Big\{I_0\Big(\tau
t\sqrt{1-x^2}\Big)+\frac{1+(2p-1)x}{\sqrt{1-x^2}}I_1\Big(\tau
t\sqrt{1-x^2}\Big)\Big\}1_{[-1,1]}(x)
\]
and $\delta_1(dx)$ (resp. $\delta_{-1}(dx)$) is the Dirac measure at $1$ (resp. $-1$).
\end{prop}
The presence of $\epsilon$ corresponds to a randomization of the
initial condition.\\ In the particular case $p=1/2$, $x\to f(t,x)$ is
symmetric and was already presented by G.H. Weiss
(\cite{weiss02} p.393) using the Fourier-Laplace transform. A
formulation of Proposition \ref{bessel} in terms of Laplace
transform will be given in Remark \ref{remark} below. Here we
shall present a proof of Proposition \ref{bessel} based
on the knowledge of conditional densities (see Proposition
\ref{loi cond}).
\begin{proof} Let $Z_t:=\epsilon Z_t^0/t$.
Let $\varphi:\Rset\to\Rset_+$ be a Borel function and
$\Delta:=\E[\varphi(Z_t)]$.
We  have: $\Delta=\sum_{n\ge 0} \Delta_n$ where $\Delta_n=\E[\varphi(Z_t)1_{\{N^0_t=n\}}]$.\\
{\bf 1)} Let $n=0$. Since $N_t^0=0\Rightarrow N_s^0=0\ \forall s\in[0,t]$, 
and $\epsilon$ and $N_t^0$ are independent, we get
\begin{eqnarray}
\label{der*}
\Delta_0&=&\E[\varphi(\epsilon)1_{\{N_t^0=0\}}]=
(\varphi(1)p+\varphi(-1)(1-p))\P(N_t^0=0)\nonumber\\
&=&(\varphi(1)p+\varphi(-1)(1-p))e^{-\tau t}.
\end{eqnarray}
{\bf 2)} Assume $n=2k\ge 1$. Using the fact that $\epsilon$ and $(N_u^0)$ are 
independent and Proposition \ref{loi cond}, we obtain:
\begin{eqnarray*}
\Delta_{2k}&=&p\E\Big[\varphi\Big(\frac{Z_t^0}{t}\Big)1_{\{N_t^0=2k\}}\Big]
+(1-p)\E\Big[\varphi\Big(-\frac{Z_t^0}{t}\Big)1_{\{N_t^0=2k\}}\Big]\\
&=& e^{-\tau t}\frac{(\tau t)^{2k}}{(2k)!}\int_{-t}^t\Big\{ p\varphi(\frac{y}{t})f_{2k}(t,y)
+(1-p)\varphi(-\frac{y}{t})f_{2k}(t,y) \Big\}dy\\
&=&e^{-\tau t}\frac{\tau^{2k} t^{2k+1}}{(2k)!}\int_{-1}^1\varphi(x)\Big\{ pf_{2k}(t,xt)
+(1-p)f_{2k}(t,-xt) \Big\}dx\\
&=& e^{-\tau t}\int_{-1}^1\varphi(x) A_{2k}(t,x)dx
\end{eqnarray*}
where
\[
A_{2k}(t,x)=\frac{1}{k!(k-1)!}\Big(\frac{\tau t}{2}\Big)^{2k}(1+(2p-1)x)(1-x^2)^{k-1}.
\]
By the same way, when $k\ge 0$, we have:
\[
\Delta_{2k+1}=e^{-\tau t}\int_{-1}^1\varphi(x) A_{2k+1}(t,x)dx,
\]
with:
\[
A_{2k+1}(t,x)=\frac{1}{(k!)^2}\Big(\frac{\tau t}{2}\Big)^{2k+1} (1-x^2)^k.
\]
Proposition \ref{bessel} is a direct consequence of \eqref{der*} and the identities:
\[
\sum_{k\ge 1}A_{2k}(t,x)=\frac{\tau t}{2}\frac{1}{\sqrt{1-x^2}}(1+(2p-1)x) I_1(\tau t\sqrt{1-x^2})
\]
\[
\sum_{k\ge 0}A_{2k+1}(t,x)=\frac{\tau t}{2}I_0(\tau t\sqrt{1-x^2}).
\]
\end{proof}
We are now able to deal with the two dimensional process $(Z_t^0, N_t^0,\, t\ge 0)$.
\begin{prop}
\label{ajout} 1) $(Z_t^0,N_t^0;\ t\ge 0)$ is a
$\Rset\times\Nset$-valued Markov process.\\
2) Let $t>s\ge 0$ and $n,m\ge 0$. Conditionally on $Z_s^0=x$ and
$N_s^0=n$, $(Z_t^0,N_t^0)$ is distributed as
\[
\Big( x+(-1)^n\int_0^{t-s}(-1)^{N_u^0}du, N_{t-s}^0 \Big)
\]
In particular,
\begin{eqnarray}
\label{eq:ajout} \P(Z_t^0-Z_s^0\in dy,\,
N_t^0-N_s^0=m)=\left\{\begin{array}{ll} \delta_{(-1)^m(t-s)}(dy) &
\mbox{if}\ m=0\\[8pt]
f_m(t-s,(-1)^ny)dy & \mbox{otherwise}.
\end{array}\right.
\end{eqnarray}
\end{prop}
\begin{proof}
Let $t>s\ge 0$. Using \eqref{defx} we get
\[
Z_t^0=Z_s^0+(-1)^{N_s^0}\int_0^{t-s}(-1)^{N_u^{0,s}}du,
\]
where $N_u^{0,s}=N_{s+u}^0-N_s^0$, $u\ge 0$.\\
Recall that $(N_u^{0,s};\ u\ge 0)\overset{(d)}{=}(N_u^0;\ u\ge 0)$
and $(N_u^{0,s};\ u\ge 0)$ is independent from $(N_u^0;\ u\le s)$.
This shows that $(Z_t^0,N_t^0;\ t\ge 0)$ is Markov. Obviously
\eqref{eq:ajout} is a direct consequence of Proposition \ref{loi
cond}.
\end{proof}
Let us now present a link between the ITN process and the
telegraph
equation.\\
Let $f:\Rset\to\Rset$ be a function of class $\mathcal{C}^2$ whose
first and second derivatives are bounded. For $c>0$, we define
(see \cite{eckstein00})
\[
u(x,t)=\frac{1}{2}\Big\{f(x+ct)+f(x-ct)\Big\}.
\]
Then $u$ is the unique solution of the wave equation
\begin{eqnarray*}
\left\{\begin{array}{l}
\displaystyle\frac{\partial^2 u}{\partial t^2}=c^2\frac{\partial^2 u}{\partial x^2},\\
\displaystyle u(x,0)=f(x),\quad \frac{\partial u}{\partial
t}(x,0)=0.
       \end{array}\right.
\end{eqnarray*}
\begin{prop}
\label{telegraph} The function
\[
w(x,t)=\E\Big[ u\Big(x,\int_0^t (-1)^{N^0_s}ds\Big)\Big],\quad
(x\in\Rset, t\ge 0)
\]
where $(Z_t^0)$ is the ITN process defined by \eqref{defx} is the
solution of the telegraph equation (TE)
\begin{eqnarray*}
\left\{\begin{array}{l}
\displaystyle\frac{\partial^2 w}{\partial t^2}+2\tau\frac{\partial w}{\partial t}
=c^2\frac{\partial^2 w}{\partial x^2},\\[8pt]
\displaystyle w(x,0)=f(x),\quad \frac{\partial w}{\partial
t}(x,0)=0.
       \end{array}\right.
\end{eqnarray*}
\end{prop}
This result can be proved using asymptotic analysis applied to
difference equation associated with the persistent random walk
\cite{goldstein50} or using Fourier transforms \cite{weiss02}.
Here we shall present a new method of proof.
\begin{proof} Since $u(x,t)$ is continuously differentiable, we
write
\begin{equation}
\label{derivation} u\Big(x,\int_0^t (-1)^{N^0_s}
ds\Big)=u(x,0)+\int_0^t (-1)^{N^0_s}\frac{\partial u}{\partial
t}\Big(x,\int_0^s (-1)^{N^0_u}du\Big)ds.
\end{equation}
Taking the average and the derivative and using the Lebesgue
theorem we get:
\begin{equation}
\label{ajout2} \frac{\partial w}{\partial
t}(x,t)=\frac{\partial}{\partial t}\E\left[u\Big(x,\int_0^t
(-1)^{N^0_s} ds\Big)\right]=\E\left[(-1)^{N^0_t}\frac{\partial
u}{\partial t}\Big(x,\int_0^t (-1)^{N^0_s}ds\Big)\right].
\end{equation}
In order to take the second $t$-derivative, we compute the
increment $A_h$ of the right hand side of the previous identity:
\[
A_h=\E\left[(-1)^{N^0_{t+h}}\frac{\partial u}{\partial
t}\Big(x,\int_0^{t+h}
(-1)^{N^0_s}ds\Big)\right]-\E\left[(-1)^{N^0_t}\frac{\partial
u}{\partial t}\Big(x,\int_0^t (-1)^{N^0_s}ds\Big)\right].
\]
This increment can be decomposed as follows: $A_h=B_h+C_h$ where
\[
B_h=\E\left[(-1)^{N^0_{t+h}}\frac{\partial u}{\partial
t}\Big(x,\int_0^{t+h}
(-1)^{N^0_s}ds\Big)\right]-\E\left[(-1)^{N^0_{t+h}}\frac{\partial
u}{\partial t}\Big(x,\int_0^t (-1)^{N^0_s}ds\Big)\right],
\]
\[
C_h=\E\left[\Big[(-1)^{N^0_{t+h}}-(-1)^{N^0_{t}}\Big]\frac{\partial
u}{\partial t}\Big(x,\int_0^{t} (-1)^{N^0_s}ds\Big)\right].
\]
First let us study $B_h$. Using equation \eqref{derivation} where
$u$ is replaced by $\frac{\partial u}{\partial t}$ and taking the
expectation we get:
\[
B_h=\E\left[(-1)^{N^0_{t+h}}\int_t^{t+h}(-1)^{N^0_s}\frac{\partial^2
u}{\partial t^2} \Big(x,\int_0^{s} (-1)^{N^0_u}du\Big)ds\right].
\]
Since the second derivative of $u$ is bounded on any compact set,
and $s\to N_s^0$ is right continuous, the dominated Lebesgue
theorem implies:
\[
\lim_{h\to 0}\frac{1}{h}B_h=\E\Big[(-1)^{2N^0_t}\frac{\partial^2
u}{\partial t^2} \Big(x,\int_0^{t} (-1)^{N^0_s}ds\Big) \Big]
=\E\Big[\frac{\partial^2 u}{\partial t^2} \Big(x,\int_0^{t}
(-1)^{N^0_s}ds\Big) \Big].
\]
Since $u$ satisfies the wave equation,
\[
\lim_{h\to 0}\frac{1}{h}B_h=c^2\E\Big[\frac{\partial^2 u}{\partial
x^2} \Big(x,\int_0^{t} (-1)^{N^0_s}ds\Big)
\Big]=c^2\frac{\partial^2 }{\partial x^2}
\E\Big[u\Big(x,\int_0^{t} (-1)^{N^0_s}ds\Big)
\Big]=c^2\frac{\partial^ 2 w}{\partial x^2}(x,t).
\]
Second, we deal with $C_h$. Since the increments of the Poisson
process are independent we have:
\begin{eqnarray*}
C_h&=&\E\left[(-1)^{N^0_t}\frac{\partial u}{\partial
t}\Big(x,\int_0^{t}
(-1)^{N^0_s}ds\Big)\right]\Big(\E[(-1)^{N^0_{t+h}-N^0_t}]-1\Big)\\
&=&\frac{\partial w}{\partial t}(x,t)(e^{-\tau h}-1-\tau h
e^{-\tau h} +o(h))=\frac{\partial w}{\partial
t}(x,t)(-2h\tau+o(h)).
\end{eqnarray*}
As a result, taking the $t$-derivative in \eqref{ajout2} comes to
\[
\frac{\partial ^2 w}{\partial t^2}(x,t)=\lim_{h\to
0}\frac{A_h}{h}=\lim_{h\to 0}\frac{B_h+C_h}{h}=c^2\frac{\partial^
2 w}{\partial x^2}(x,t)-2\tau\frac{\partial w}{\partial t}(x,t).
\]
The function $w$ is actually the solution of the telegraph
equation. It is easy to prove that $w$ satisfies the boundary
conditions.\end{proof}}{}}
\section{Convergence of the persistent walk to the ITN}
\label{sectionpreuve1} Suppose $\rho_0=1$. The aim of this section is to prove the
convergence of the interpolated persistent random walk towards the
generalized integrated telegraph noise (ITN) i.e. Theorem \ref{cas1}. 
Let us start with preliminary results.\\
First, let us recall that $(X_n, n\in\Nset)$ is the persistent
random walk starting in $0$ defined by the increments process
$(Y_n,\,n\in\Nset)$ (see Section \ref{section_not}) with transition probabilities
\begin{eqnarray*}
\pi^\Delta=\left(\begin{array}{cc}
1-c_0\Delta_t & c_0\Delta_t\\
c_1\Delta_t & 1-c_1\Delta_t
      \end{array}\right).
\end{eqnarray*}
Let $(T_k;\ k\ge 1)$ be the sign changes sequence of times :
\begin{eqnarray}
\label{*1} \left\{\begin{array}{l}
T_1=\inf\{t\ge 1 \ : Y_t\neq Y_0\}\\[8pt]
T_{k+1}=\inf\{t> T_k\ : Y_t\neq Y_{T_{k}}\};\quad k\ge 1.
       \end{array}\right.
\end{eqnarray}
We put  $T_0=0$ and
\begin{equation} \label{*2}
A_k=T_k-T_{k-1}\quad k\ge 1
\end{equation}
Let  $N_t$ be the number of times over $[0,t]$ so that the sign of
 $(Y_n)$ changes:
\begin{equation}
\label{*3} N_t=\sum_{j\ge 1}1_{\{T_j\le t\}}
\end{equation}
The definition of $N_t$ implies that:
\[
N_t=k\Longleftrightarrow T_k\le t<T_{k+1}
\]
We suppose in this subsection that $Y_0=-1$.\\
We deduce from the identities above:
\begin{equation}
\label{*4} X_t=\sum_{j=1}^k (-1)^j
A_j+(-1)^{k+1}(t-T_k+1)\quad\mbox{where}\ k=N_t.
\end{equation}
By \eqref{*2} we obtain:
\begin{equation}
\label{*5} T_k=A_1+\ldots+A_k\quad k\ge 1.
\end{equation}
Hence the equations \eqref{*3}, \eqref{*4} and \eqref{*5} permit
to emphasize the bijective correspondence between  $(X_n;\
n\in\Nset)$ and $(A_k;\ k\in\Nset)$.\\
We introduce the normalization of $(X_n;\ n\in\Nset)$ given by
\eqref{scal} with $\Delta_x=\Delta_t$:
\begin{equation}
\label{4.4b}
Z^\Delta_s=\Delta_t X_{s/\Delta_t}\quad (s/\Delta_t\in\Nset).
\end{equation}
Let us define:
\begin{equation}
\label{*6} N^\Delta_s=\sum_{j\ge 1}1_{\{\sum_{i=1}^j\Delta_t
A_j\le s\}}\quad s\ge 0.
\end{equation}
Let us note that
\[
N^\Delta_s=N_{s/\Delta_t}\quad\mbox{if}\ s/\Delta_t\in\Nset.
\]
That permits to extend the definition of $Z^\Delta_s$ to any $s\ge
0$ by setting
\begin{equation}
\label{*7} \tilde{Z}_s^\Delta=\sum_{j=1}^k (-1)^j(\Delta_t
A_j)+(-1)^{k+1}(s-\Delta_t T_k+\Delta_t)\quad k=N^\Delta_s.
\end{equation}
Obviously $\tilde{Z}_s^\Delta=Z^\Delta_{s}$ if $s/\Delta_t\in\Nset$.\\
In order to study the asymptotic behaviour of
$(\tilde{Z}_s^\Delta)$ as $\Delta_t\to 0$, we shall first prove
the convergence in distribution of $(\Delta_t A_j)_{j\ge 1}$ and
$(N^\Delta_s)_{s\ge 0}$. \\
We recall that some random variable $\xi$ is exponentially
distributed with parameter $\lambda>0$ if its density is given by
$\frac{1}{\lambda}\, e^{-x/\lambda}1_{\{x\ge 0\}}$.
\begin{lem}
\label{convloi} The random variables  $(A_k)$ are independent and
$\Delta_t A_{2k}$ (resp. $\Delta_t A_{2k+1}$) converges in
distribution, as $\Delta_t\to 0$, to the exponential law with
parameter $\frac{1}{c_1}$ (resp. $\frac{1}{c_0}$).
\end{lem}
\begin{proof} Since $(Y_n)$ is a Markov chain, then the $(A_k)$ are 
independent. First let us study the convergence in distribution of the
sequence $\Delta_t A_{2k}$. We use the Laplace transform of
$\Delta_t A_{2k}$: $\varphi(\mu)=\E[e^{-\mu\Delta_tA_{2k}}]$,
$\mu\ge 0$. Since $A_{2k}$ is geometrically distributed with
parameter $c_1\Delta_t$, we obtain
\begin{eqnarray}
\label{25**2}
\varphi(\mu)&=&\sum_{j=1}^\infty e^{-\mu\Delta_t j}(1-c_1\Delta_t)^{j-1}c_1\Delta_t\nonumber\\
&=&
\frac{c_1\Delta_t}{e^{\mu\Delta_t}-(1-c_1\Delta_t)}=\frac{c_1\Delta_t}{(\mu+c_1)\Delta_t
+o(\Delta_t)}=\frac{c_1}{\mu+c_1}+o(\Delta_t)
\end{eqnarray}
The function $\varphi(\mu)$ converges for any $\mu\ge 0$ to the
Laplace transform of some exponential law with parameter
$c_1^{-1}$. This proves the convergence in distribution of
$\Delta_t A_{2k}$. Concerning $A_{2k-1}$ the arguments are
similar.\end{proof} Let us recall that the counting process
$(\N_t, \ t\ge 0)$ has been defined through the sequence of jumps
$(e_n; n\ge 1)$ via \eqref{**1.1}, and $(e_n;\ n\ge 1)$ are i.i.d. and exponentially distributed.
\begin{lem}
\label{lem:convmarg}
Let $s>0$, $k\ge 1$ and $\Phi_k:\Rset^{k}\to\Rset$ be a bounded continuous function. Then\\[5pt]
1) $\displaystyle\lim_{\Delta_t\to 0}\P(N^\Delta_s=0)=\P(\N_s=0)$\\
2) $\displaystyle\lim_{\Delta_t\to
0}\E[\Phi_k(\Delta_tA_1,\Delta_tA_2,\ldots,\Delta_tA_k)1_{\{N^{\Delta}_s=k\}}]=
\E[\Phi_k(\lambda_1e_1,\lambda_2e_2,\ldots,\lambda_ke_k)1_{\{\N_s=k\}}],$
where $\lambda_k$ has been defined by \eqref{.*.}.
\end{lem}
\begin{proof}
1) Statement 1) follows from:
\begin{eqnarray*}
\P(N^\Delta_s=0)&=&\P(N_{\lfloor s/\Delta_t\rfloor}=0)=\P(T_1\ge \lfloor s/\Delta_t\rfloor)\\
&=&\P(A_1\ge \lfloor s/\Delta_t\rfloor)=\P(\Delta_tA_1\ge
\Delta_t\lfloor s/\Delta_t\rfloor)
\end{eqnarray*}
where $\lfloor a \rfloor$ denotes the integer part of $a$.\\
2) Set $k\ge 1$. The event $\{N^{\Delta}_s=k\}$ can be decomposed
as follows:
$$
\{N^{\Delta}_s=k\}=\left\{\Delta_t \sum_{j=1}^k A_j\le
s\right\}\cap\left\{\Delta_t \sum_{j=1}^{k+1} A_j> s\right\}.
$$
This identity imply existence of a bounded Borel function
$\psi_k$\,:\,$\Rset^{k+1}\to\Rset$ so that
\begin{eqnarray*}
&&\Phi_k(\Delta_tA_1,\ldots,\Delta_tA_k)1_{\{N^{\Delta}_s=k\}}\\
&&=\Phi_k(\Delta_tA_1,\ldots,\Delta_tA_k)1_{\{\Delta_t \sum_{j=1}^k A_j\le s\}}
1_{\{ \Delta_t \sum_{j=1}^{k+1} A_j> s\}}\\
&&=\psi_k(\Delta_tA_1,\Delta_tA_2,\ldots,\Delta_tA_{k+1}).
\end{eqnarray*}
Since $\Phi_k$ is continuous, the discontinuity points of $\psi_k$
are included in:
\[
\mathbb{U}=\Big\{x\in\Rset^{k+1}:\ \sum_{j=1}^k x_j=s
\Big\}\cup\Big\{x\in\Rset^{k+1}:\ \sum_{j=1}^{k+1} x_j=s \Big\}.
\]
By Lemma \ref{convloi}, $(\Delta_t A_1,\ldots,\Delta_t A_{k+1})$
converges in distribution towards $(\lambda_1 e_1,\ldots,\lambda_{k+1}e_{k+1})$ as
$\Delta_t\to 0$. Since the Lebesgue measure of $\mathbb{U}$ is
null, the limit law does not charge $\mathbb{U}$. We can conclude
evoking for instance Theorem 14 p.247 in
\cite{Brancovan06}).\end{proof}
Let us formulate a straightforward generalization of Lemma
\ref{lem:convmarg}.
\begin{lem}
\label{lem:convmarggen} Let $n\in \Nset$,
$(k_1,\ldots,k_n)\in\Nset^n$ such that $k_1\le k_2\le\ldots\le
k_n$ and $(s_1,\ldots,s_n)\in\Rset_+^n$ with $s_1\le
s_2\le\ldots\le s_n $. Let $\Phi:\Rset^{k_n}\to\Rset$ be a bounded
and continuous function. Then
\begin{eqnarray}
\label{conv_marg_gen} &&\lim_{\Delta_t\to
0}\E[\Phi(\Delta_tA_1,...,\Delta_tA_{k_n})1_{\{N^{\Delta}_{s_1}=k_1,...,N^{\Delta}_{s_n}=k_n
\}}]\nonumber\\
&&\quad =\E[\Phi(\lambda_1e_1,...,\lambda_{k_n}e_{k_n})1_{\{\N_{s_1}=k_1,...,\N_{s_n}=k_n\}}]
\end{eqnarray}
\end{lem}
\begin{prop}
\label{convmarg} The random variable $\tilde{Z}^\Delta_s$
converges in distribution towards $-\Z_s$, for any $s>0$, as
$\Delta_t\to 0$.
\end{prop}
\begin{proof} Let $f\ :\ \Rset \to\Rset$ be a continuous function which 
is bounded by $M$. Identities \eqref{*7} and \eqref{*5} imply that $
\E[f(\tilde{Z}^\Delta_s)]=\sum_{k=0}^\infty E_\Delta(k)$, with
\[
E_\Delta(k)=\E\Big[f\Big(\sum_{j=1}^k
(-1)^j\Delta_tA_j+(-1)^{k+1}\Big(s-\Delta_t\sum_{j=1}^k A_j+\Delta_t\Big)
\Big)1_{\{N^\Delta_s=k\}} \Big]
\]
Applying Lemma \ref{lem:convmarg} and \eqref{resum__*}, we obtain for any $k\ge 0$,
\begin{eqnarray*}
\lim_{\Delta_t\to 0}E_\Delta(k)&=&E\Big[f\Big(\sum_{j=1}^k
(-1)^j\lambda_j e_j+(-1)^{k+1}\Big(s-\sum_{j=1}^k \lambda_j e_j\Big)
\Big)1_{\{\N_s=k\}} \Big]\\
&=&\E[f(-\Z_s)1_{\{\N_s=k \}}].
\end{eqnarray*}
Moreover since $f$ is bounded by $M$, we get
\[
\vert E_\Delta(k)\vert\le M\P(N^\Delta_s=k).
\]
Suppose that $k\ge 1$. Then, using the Markov inequality and the
independence property of the random sequence $(A_n,\ n\ge 0)$, we
obtain
\begin{eqnarray*}
\P(N^\Delta_s=k)&=&\P\Big(\Delta_t\sum_{j=1}^k A_j\le
s<\Delta_t\sum_{j=1}^{k+1}A_j
\Big)\\
&\le&\P\Big(\Delta_t\sum_{j=1}^k A_j\le
s\Big)=\P\Big(\exp\Big\{-\Delta_t\sum_{j=1}^k A_j\Big\}\ge
e^{-s}\Big)\\
&\le&e^s\E\Big[\exp-\Delta_t\sum_{j=1}^k A_j \Big]=e^
s\prod_{j=1}^k\varphi_j(1)
\end{eqnarray*}
where $\varphi_j(\mu)=\E[e^{-\mu\Delta_tA_j}]$. Since $(Y_n)$ is a
Markov chain starting at $Y_0=-1$, for any $j\ge 1$, $A_{2j-1}$
(resp. $A_{2j}$) is geometrically distributed with parameter
$c_0\Delta_t$ (resp. $c_1\Delta_t$). According to \eqref{25**2} we
get
\[
\varphi_{2j}(1)=\frac{c_1\Delta_t}{e^{\Delta_t}-1+c_1\Delta_t}\le
\frac{c_1\Delta_t}{\Delta_t+c_1\Delta_t}=\frac{c_1}{1+c_1}<1.
\]
By the same way, we have:
\[
\varphi_{2j-1}(1)\le\frac{c_0}{1+c_0}<1.
\]
As a result, there exists $0<r<1$ so that
\begin{equation}
\label{lebesgue} \P(N^\Delta_s=k)\le e^s r^k.
\end{equation}
We are now allowed to apply the dominated convergence theorem:
\[
\lim_{\Delta_t\to 0}\E[f(\tilde{Z}^\Delta_s)]=\sum_{k\ge
0}\lim_{\Delta_t\to 0}E_\Delta(k)=\sum_{k\ge
0}\E[f(-\Z_s)1_{\{\N_s=k\}}]=\E[f(-\Z_s)].
\]\end{proof}
\begin{prop}
\label{convmarg_gen} For any $(s_1,\ldots,s_n)\in\Rset_+^n$ such
that $s_1\le s_2\le\ldots\le s_n $, the random vector $
(\tilde{Z}^\Delta_{s_1},\ldots,\tilde{Z}^\Delta_{s_n})$ converges
in distribution to  $(-\Z_{s_1},\ldots,-\Z_{s_n})$,  as $\Delta_t$
tends to $0$.
\end{prop}
\begin{proof} We follow the approach developed in the proof of Proposition 
\ref{convmarg}. Let $f:\Rset^n\to\Rset$ be a bounded and continuous function. We have:
\[
\E\Big[f(\tilde{Z}^\Delta_{s_1},\ldots,\tilde{Z}^\Delta_{s_n})\Big]=\sum_{k_1,\ldots,k_n}
E_\Delta(k_1,\dots,k_n),
\]
where the sum is extended to $(k_1,\ldots,k_n)\in\Nset^n$ so that
$k_1\le k_2\le\ldots\le k_n$ and
\[
E_\Delta(k_1,\dots,k_n)=\E\Big[f(\tilde{Z}^\Delta_{s_1},\ldots,\tilde{Z}^\Delta_{s_n})1_{\{
N^{\Delta}_{s_1}=k_1,\ldots,N^{\Delta}_{s_n}=k_n\}}\Big].
\]
Identity \eqref{*7} implies the existence of a bounded continuous
function $\psi_n:\Rset^{k_n}\to\Rset$ so that
\[
E_\Delta(k_1,\dots,k_n)=\E\Big[\psi_n(\Delta_tA_1,\ldots,\Delta_tA_{k_n})1_{\{
N^{\Delta}_{s_1}=k_1,\ldots,N^{\Delta}_{s_n}=k_n\}}\Big].
\]
Applying Lemma \ref{lem:convmarggen}, we get
\[
\lim_{\Delta_t\to
0}E_\Delta(k_1,\dots,k_n)=\E[\psi_{n}(\lambda_1 
e_1,\ldots,\lambda_{k_n}e_{k_n})1_{\{\N_{s_1}=k_1,\ldots,\N_{s_n}=k_n\}}].
\]
According to the definition of the process $\Z_s$, we may deduce
:
\[
\lim_{\Delta_t\to 0}E_\Delta(k_1,\dots,k_n)=\E[f(-\Z_{s_1},\ldots,-\Z_{s_n})
1_{\{\N_{s_1}=k_1,\ldots,\N_{s_n}=k_n\}}].
\]
In order to obtain that
\[
\lim_{\Delta_t\to
0}\E[f(\tilde{Z}^\Delta_{s_1},\ldots,\tilde{Z}^\Delta_{s_n})]=\E[f(-\Z_{s_1},\ldots,-\Z_{s_n}
)],
\]
it suffices (cf the proof of Proposition \ref{convmarg}) to prove
that
\[
\sum_{k_1,\ldots, k_{n-1}}\sup_{\Delta_t}\vert
E_\Delta(k_1,\dots,k_n)\vert<\infty.
\]
Since $f$ is bounded,
\[
\vert E_\Delta(k_1,\dots,k_n)\vert\le M\P(N_{s_n}^\Delta=k_n)
\]
Using moreover \eqref{lebesgue} we get
\[
\sum_{k_1,\ldots, k_{n}}\vert E_\Delta(k_1,\dots,k_n)\vert\le M
e^{s_n}\sum_{k_n}(k_n)^{n-1}r^{k_n}<\infty
\]
since $r<1$.\end{proof} We are now able to complete the proof of Theorem \ref{cas1}. 
Since $(\tilde{Z}^\Delta_s)$ and $(\Z_s)$
are both continuous processes, the convergence of the process
$(\tilde{Z}^\Delta_s)$ to the process $(-\Z_s)$ will be proved as
soon as the following measure tension criterium (cf Theorem 8.3
p.56 in \cite{Billingsley99}) holds :  for all $\eps>0$ and
$\eta_0$, there exists some constants $\delta\in]0,1[$ and $\mu>0$
such that
\begin{equation}\label{tension}
\frac{1}{\delta} \P\Big(\sup_{s\le u\le s+\delta}\vert
\tilde{Z}^\Delta_{u}-\tilde{Z}^\Delta_s\vert\ge\eps
\Big)\le\eta_0,\quad\mbox{for any}\ \Delta_t\le\mu.
\end{equation}
Since $(\tilde{Z}^\Delta_s,\ s\ge 0)$ is the interpolated
persistent random walk, its slope is always equal to $1$ or $-1$.
Hence we obtain for any $(u,s)\in\Rset_+^2$,
\[
\vert \tilde{Z}^\Delta_u-\tilde{Z}^\Delta_s\vert \le\vert
u-s\vert.
\]
Consequently
\[
\sup_{s\le u\le s+\delta}\vert
\tilde{Z}^\Delta_u-\tilde{Z}^\Delta_s\vert \le \delta.
\]
By choosing $\delta=\eps/2$ we get the tension criterium and so
the convergence of the process $(\tilde{Z}^\Delta_s)$ to the
process $(-\Z_s)$.
\section{Two extensions of Theorem \ref{cas1}}
\label{extensions}
First of all, the extensions presented in this section concerns 
the regime $\Delta_x=\Delta_t$.
\subsection{The case when $(Y_t)$ takes $k$ values.}
Let us introduce our parameters. Let $k\ge 2$, $y_1,\ldots,y_k$ denote $k$ real 
numbers, and $(c(i,j);\ 1\le i,j\le k)$ a matrix so that 
\begin{equation}
\label{cadre}
c(i,j)\ge 0\ \mbox{for any}\ i\neq j,\ c(i,i)=0,\ \sum_{l=1}^k c(i,l)>0\ \forall i.
\end{equation}
We directly consider the asymptotic regime. Let $(Y_t)$ be a 
$\{y_1,\ldots, y_k\}$-valued Markov chain, with transition probability matrix:
\begin{eqnarray}
\label{gen:matrix}
\pi^\Delta(y_i,y_j)=\left\{\begin{array}{ll}
c(i,j)\Delta_t & i\neq j\\[5pt]
1-\Big( \sum_{l=1}^k c(i,l) \Big)\Delta_t & i=j,
                         \end{array}\right.
\end{eqnarray}
where $\Delta_t>0$ is supposed to be small so that
\[
c(i,j)\Delta_t\le 1,\quad \Big( \sum_{l=1}^k c(i,l) \Big)\Delta_t<1.
\]
Similarly to the case $k=2$ and $y_1=-1$, $y_2=1$, we are interested in 
the linear interpolation $(\tilde{Z}_s^\Delta;\ s\ge 0)$ of the process 
$(Z^\Delta_s;\ s\ge 0)$ defined by \eqref{scal}.
\begin{thm}
\label{the:ext1}
Suppose $Y_0=y_i$. Then $(\tilde{Z}^\Delta_s;\ s\ge 0)$ converges in 
distribution, as $\Delta_t\to 0$, to the process $\Big( \int_0^t R_sds;\ t\ge 0\Big)$ 
where $(R_s)$ is a $\{y_1,\ldots, y_k\}$-valued continuous-time Markov chain 
starting at level $y_i$, whose dynamic is the following: $(R_t)$ stays on 
level $y_i$ an exponential time with parameter $1/\Big( \sum_{l=1}^k c(j,l) \Big)$ 
and jumps to $y_{j'}$ ($j'\neq j$) with probability 
$c(j,j')/\Big( \sum_{l=1}^k c(j,l) \Big)$.
\end{thm}
\begin{rem}
In the case $k=2$, $y_1=-1$ and $y_2=1$, then $((-1)^{\N_t};\ t\ge 0)$ 
(cf \eqref{**1.1}) may be chosen as a realization of $(R_t)$ when it starts at $R_0=-1$.
\end{rem}
\noindent
{\it Proof of Theorem \ref{the:ext1}.}
We proceed as in the proof of Theorem \ref{cas1} developed in Section \ref{sectionpreuve1}.\\
Let $(T_n)_{n\ge 1}$ be the sequence of stopping times defined by \eqref{*1}. Then:
\begin{eqnarray*}
X_t=\left\{\begin{array}{ll}
y_i(t+1) & 0\le t<T_1\\[5pt]
y_i T_1+Y_{T_1}(t-T_1+1) & T_1\le t <T_2.
           \end{array}\right.
\end{eqnarray*}
Recall that $(Z_s^\Delta;\ s/\Delta_t\in\Nset)$ has been defined by \eqref{*5}. 
From the relations above, it is easy to deduce:
\begin{eqnarray*}
Z_s^\Delta=\left\{\begin{array}{l}
y_i(s+\Delta_t)\quad 0\le s\le \Delta_t T_1\\[5pt]
y_i(\Delta_t T_1)+Y_{T_1}(s-\Delta_t T_1+\Delta_t)\quad \Delta_t T_1\le s<\Delta_t T_2.
                  \end{array}\right.
\end{eqnarray*}
Let us determine the limit distribution of $(\Delta_t T_1, Y_{T_1})$ as $\Delta_t\to 0$. Set
\[
V^\Delta(\lambda, j)=\E\Big[ e^{-\lambda \Delta_t T_1}1_{\{ Y_{T_1}=y_{j} \}} \Big],
\quad \lambda>0,\ j\neq i.
\]
Proceeding as in the proof of Lemma \ref{convloi}, we obtain:
\[
V^\Delta(\lambda, j)=\frac{e^{-\lambda\Delta_t} c(i,j)\Delta_t}{1-\Big[1-\Big( \sum_{l=1}^k c(i,l)\Big)\Delta_t\Big]e^{-\lambda\Delta_t}}.
\]
Using standard analysis, we deduce that $(\Delta_t T_1, Y_{T_1})$ converges in 
distribution as $\Delta_t\to 0$ to $(e'_1, U_1)$ where:
\[
\E\Big[ e^{-\lambda e'_1}1_{\{ U_1=j\}} \Big]=\frac{c(i,j)}{\lambda+\sum_{l=1}^k c(i,l)}.
\]
As a result, $e'_1$ and $U_1$ are independent, $e'_1$ is exponentially distributed 
with parameter $1/\sum_{l=1}^k c(i,l)$ and
\[
\P(U_1=j)=\frac{c(i,j)}{\sum_{l=1}^k c(i,l)}.
\]
Using the approach developed in Section \ref{sectionpreuve1}, we can 
prove Theorem \ref{the:ext1}. The details are left to the reader.
\hfill{$\Box$}
\subsection{The case when $(Y_t)$ is a Markov chain of order $2$.}
\label{ext2}
Let $(Y_t)$ be a Markov chain with order $2$. For simplicity we suppose that 
it takes its values in $\{-1,1\}$. Obviously $(Y_t,Y_{t+1})_{t\ge 0}$ is a 
Markov chain with state space 
\[
E=\{(-1,-1),(-1,1),(1,-1),(1,1)\}.
\] 
Let $\pi^\Delta$ be the transition probability matrix:
\begin{eqnarray}
\label{matrixx}
\pi^\Delta=\left(\begin{array}{cccc}
1-c_0\Delta_t & c_0\Delta_t & 0 & 0\\
0 & 0 & 1-p_0 & p_0\\
p_1 & 1-p_1 & 0 & 0\\
0 & 0 & c_1 \Delta_t & 1-c_1\Delta_t
                 \end{array}\right)
\end{eqnarray}
where $\Delta_t,c_0,c_1,p_0,p_1>0$ and $c_0\Delta_t, c_1\Delta_t, p_0, p_1<1$. \\
Let us introduce:
\begin{equation}
\label{defv}
v_i=\frac{p_i}{1-(1-p_0)(1-p_1)},\quad c'_i=c_iv_i,\ i=0,1.
\end{equation}
Recall that $(Z^\Delta_t)$ and $(\tilde{Z}^\Delta_t)$ have been defined by \eqref{4.4b}, 
resp. \eqref{*7}, $(\N_t)$ is the counting process defined by \eqref{**1.1}, and
\[
\Z_t=\int_0^t (-1)^{\N_u}du,\ t\ge 0.
\]
\begin{thm}
\label{the:ext2}
1) Suppose that $Y_0=Y_1=-1$ (resp. $Y_0=Y_1=1$) then $(\tilde{Z}^\Delta_s;\ s\ge 0)$ 
converges in distribution, as $\Delta_t\to 0$, to $(-Z^{c'_0,c'_1}_s;\ s\ge 0)$ 
(resp. $(Z^{c'_1,c'_0}_s;\ s\ge 0)$).\\
2) Suppose $Y_0=1$ and $Y_1=-1$ (resp. $Y_0=-1$, $Y_1=1$) then 
$(\tilde{Z}^\Delta_s;\ s\ge 0)$ converges in distribution, as $\Delta_t\to 0$, to 
\[
\Big((\epsilon-1)\int_0^s(-1)^{\NN_u}du+\epsilon\int_0^s(-1)^{\nn_u}du;\ s\ge 0  \Big)
\]
where $\epsilon$ is independent from $(\NN_u)$, $(\nn_u)$ and 
\[
\P(\epsilon=0)=1-\P(\epsilon=1)=v_1\quad (\mbox{resp.}\ \P(\epsilon=1)=1-\P(\epsilon=0)=v_0).
\]
\end{thm}
\begin{rem}
1) Note that $(Y_t)_{t\in\Nset}$ is a Markov chain if and only if $1-c_0\Delta_t=p_1$ 
and $1-c_1\Delta_t=p_0$. If we replace formally $p_0$ (resp. $p_1$) by $1-c_1\Delta_t$ 
(resp. $1-c_0\Delta_t$) in \eqref{defv} and take the limit $\Delta_t\to 0$, we obtain 
$v_i=p_i$ and $c'_i=c_i$. We recover Theorem \ref{cas1}.\\
2) The fact that $(Y_t)$ is a Markov chain with order $2$ does not modify drastically 
the limit. The limit process can be expressed in terms of processes of the 
type $(Z_s^{\alpha,\beta};\ s\ge 0)$.
\end{rem}
\noindent
{\it Proof of Theorem \ref{the:ext2}.} {\bf 1)} We only consider the case $Y_0=Y_1=1$. 
Let us define $T_1$, $T_2$ and $T_3$ as follows:
\[
T_1=\inf\{ t\ge 1,\ Y_t=-1 \},\quad T_2=\inf\{ t\ge T_1+1,\ Y_t=Y_{t-1} \},
\quad T_3=\inf\{ t\ge T_2+1,\ Y_t\neq Y_{T_2} \}.
\]
Using the definition (cf \eqref{25*1}) of $(X_t)$ we easely obtain:
\begin{eqnarray*}
X_t=\left\{\begin{array}{l}
t+1\quad 0\le t<T_1\\
T_1+\hat{X}_t\quad T_1\le t<T_2
           \end{array}\right.
\end{eqnarray*}
where $\hat{X}_t$ equals either $-1$ or $0$.\\
Moreover, when $T_2\le t<T_3$, we have:
\begin{eqnarray*}
X_t=\left\{\begin{array}{l}
T_1-2-(t-T_2)\quad \mbox{if}\ T_2-T_1\ \mbox{is odd}\\
T_1+1+(t-T_2)\quad \mbox{otherwise.}
           \end{array}\right.
\end{eqnarray*}
According to \eqref{scal}, we can deduce:
\begin{eqnarray*}
Z_s^\Delta=\left\{\begin{array}{ll}
s+\Delta_t & 0\le s\le \Delta_t T_1\\[5pt]
\Delta_t T_1+\Delta_t\hat{X}_{s/\Delta_t}& \Delta_t T_1\le s<\Delta_t T_2\\[5pt]
\Delta_t T_1-2\Delta_t-(s-\Delta_t T_2) & \Delta_t T_2\le s<\Delta_t T_3,\ Y_{T_2}=-1\\[5pt]
\Delta_t T_1+\Delta_t+s-\Delta_t T_2 & \Delta_t T_2\le s<\Delta_t T_3,\ Y_{T_2}=1
                \end{array}\right.
\end{eqnarray*}
(note that $T_2-T_1$ is odd if and only if $Y_{T_2}=-1$).\\
{\bf 2) a)} Proceeding as in the proof of Theorem \ref{cas1}, we can prove 
that $\Delta_t T_1$ converges in distribution, as $\Delta_t\to 0$, to $e'_1$, 
where $e'_1$ is exponentially distributed with parameter $1/c_1$. 
Then $(\tilde{Z}_s^\Delta; \ 0\le s\le \Delta_t T_1)\stackrel{(d)}{\longrightarrow}
(s;\ s\le e'_1)$, as $\Delta_t\to 0$.\\
{\bf b)} The distribution of $T_2-T_1$ does not depend on $\Delta_t$.
 Moreover $\vert\hat{X}_{\cdot}\vert\le 1$, then the limit of 
 the length of the interval $[\Delta_t T_1,\Delta_t T_2]$ is null. We have
\[
\P( Y_{T_2}=-1)=\sum_{l\ge 0} \Big( (1-p_1)(1-p_0) \Big)^l p_1=v_1.
\]
{\bf c)} Using the strong Markov property, we easely show that 
$(\tilde{Z}_{s+\Delta_t T_2}^\Delta; \ 0\le s\le \Delta_t (T_3-T_1))\stackrel{(d)}{\longrightarrow}
(e'_1+Y_{T_1}s;\ 0\le s\le e'_2)$, as $\Delta_t\to 0$, where $(e'_1, Y_{T_1})$ 
(resp. $(e'_1,e'_2)$) are independent r.v.'s and conditionally on $Y_{T_2}=1$ 
(resp. $Y_{T_2}=-1$) $e'_2$ is exponentially distributed with parameter $1/c_1$ 
(resp. $1/c_0$).\\
{\bf d)} Let us summarize the former analysis. We have proved that  
$(\tilde{Z}_s^\Delta;\ s\ge 0)\stackrel{(d)}{\longrightarrow}(\int_0^s \hat{R}_udu,\ s\ge 0)$,
where $(\hat{R}_u)$ is a continuous-time Markov chain which takes its values in 
$\{-1,1\}$ and $\hat{R}_0=1$. Moreover the dynamic of $(\hat{R}_u)$ is the following:
$(\hat{R}_u)$ stays in $1$ (resp. $-1$) an exponential time with parameter $1/c_1$ 
(resp. $1/c_0$) and moves to $-1$ (resp. $1$) with probability $v_1$ (resp. $v_0$). 
Note that $(\hat{R}_u)$ is allowed to stay in the same site. It is classical 
(cf \cite{ross03}) to prove that $(\hat{R}_u)_{u\ge 0}\stackrel{(d)}{=}(\zz_u)_{u\ge 0}$ 
where $c'_0$ and $c'_1$ are defined by \eqref{defv}.\hfill{$\Box$}

\section{Convergence of the persistent random walk towards the Brownian motion with drift}
\label{section preuve2} In subsection \ref{mgf} below we determine
the generating function of $X_t$, where $X_t$ is the persistent
random walk defined by \eqref{25*1}. This allows to prove Theorem
\ref{conv1} and Proposition \ref{cas_limit} in subsections \ref{***}, \ref{scas_limit}.
\mathversion{bold}
\subsection{The moment generating function of $X_t$}\label{mgf}
\mathversion{normal} Let us recall that the increments process
$(Y_t,\, t\in\mathbb{N})$ is a Markov chain valued in the state
space $E=\{-1,1\}$. Its transition probability is given by
\begin{eqnarray*}
\pi=\left(\begin{array}{cc} 1-\alpha & \alpha\\
\beta & 1-\beta\end{array}\right)\quad\quad 0<\alpha<1,\quad
0<\beta<1.
\end{eqnarray*}
The persistent random walk $(X_t,\, t\in\mathbb{N})$ is defined by
the partial sum:
\[
X_t=\sum_{i=0}^t Y_i\quad\mbox{with}\quad X_0=Y_0=1\ \mbox{or}\
-1.
\]
\begin{lem}
\label{***1} Let us define the functions $a_t$ and $b_t$:
\begin{equation}
\label{eq:***1}
a_t(j)=\P(X_t=j, Y_t=-1)\quad\mbox{and}\quad b_t(j)=\P(X_t=j,
Y_t=1).
\end{equation}
Then,
\begin{equation}
\label{aj} a_{t+1}(j)=(1-\alpha)a_t(j+1)+\beta b_t(j+1)
\end{equation}
\begin{equation}
\label{bj} b_{t+1}(j)=\alpha a_t(j-1)+(1-\beta)b_t(j-1).
\end{equation}
\end{lem}
\begin{proof} Using the Markov property of $(Y_t)$ we have:
\begin{eqnarray*}
 a_{t+1}(j)&=&\P(X_{t+1}=j, Y_{t+1}=-1,
Y_t=-1)+\P(X_{t+1}=j, Y_{t+1}=-1, Y_t=1)\\
&=&\P(X_{t}=j+1, Y_{t+1}=-1, Y_t=-1)+\P(X_{t}=j+1, Y_{t+1}=-1,
Y_t=1)\\
&=& (1-\alpha)a_t(j+1)+\beta b_t(j+1).
\end{eqnarray*}
The second recursive formula involving $(b_t(j))$ can be obtained
similarly.\end{proof} Let us define the moment generating function
$\Phi(\lambda,t)=\E[\lambda^{X_t}],\quad(\lambda>0)$. We decompose
$\Phi(\lambda,t)$ as
\begin{equation}
 \label{phi}
\Phi(\lambda,t)=\Phi_-(\lambda,t)+\Phi_+(\lambda,t),
\end{equation}
with
\begin{equation}
\label{phi_decomp}
\Phi_-(\lambda,t)=\E[\lambda^{X_t}1_{\{Y_t=-1\}}],\quad\quad
\Phi_+(\lambda,t)=\E[\lambda^{X_t}1_{\{Y_t=1\}}].
\end{equation}
\begin{lem} 1) $\Phi_-(\lambda,0)=\frac{1}{\lambda}\P(Y_0=-1)$
and
$\Phi_+(\lambda,0)=\lambda\P(Y_0=1)$.\\
2) The moment generating function verifies the following induction
equations:
\begin{eqnarray}
\label{rec1} \Phi_-(\lambda,t+1)=\frac{1-\alpha}{\lambda}\
\Phi_-(\lambda,t)+\frac{\beta}{\lambda}\ \Phi_+(\lambda,t)\\
\label{rec2}\Phi_+(\lambda,t+1)=\alpha\lambda
\Phi_-(\lambda,t)+(1-\beta)\lambda \Phi_+(\lambda,t)
\end{eqnarray}
\end{lem}
\noindent \begin{proof} Definition \eqref{eq:***1} implies that
$$\Phi_-(\lambda,t)=\sum_{j\in\mathbb{Z}}\lambda^j
a_t(j)=\sum_{j=-t-1}^{t+1}\lambda^j a_t(j).$$ Hence,
\begin{eqnarray*}
\Phi_-(\lambda,t+1)&=&\sum_j \lambda^j
a_{t+1}(j)=(1-\alpha)\sum_j\lambda^j a_t(j+1)+\beta\sum_j
\lambda^j
b_t(j+1)\\
&=&(1-\alpha)\frac{1}{\lambda}\sum_j
\lambda^{j+1}a_t(j+1)+\frac{\beta}{\lambda}\sum_j
\lambda^{j+1}b_t(j+1)\\
&=&\frac{1-\alpha}{\lambda}\Phi_-(\lambda,t)+\frac{\beta}{\lambda}\Phi_+(\lambda,t).
\end{eqnarray*}
The proof of \eqref{rec2} is similar.\end{proof}
\begin{lem}\label{***3}
Let $f(\lambda,t)$ be equal to either $\Phi_-(\lambda,t)$ or $\Phi_+(\lambda,t)$, then
\begin{equation}
\label{seconddegr}
f(\lambda,t+2)-\Big(\frac{1-\alpha}{\lambda}+(1-\beta)\lambda\Big)f(\lambda,t+1)
+(1-\alpha-\beta)f(\lambda,t)=0.
\end{equation}
\end{lem}
\noindent \begin{proof} By \eqref{rec1}, we get
\begin{equation}
\label{poly}
\Phi_+(\lambda,t)=\Big\{\Phi_-(\lambda,t+1)-\frac{1-\alpha}{\lambda}\Phi_-(\lambda,t)\Big\}
\frac{\lambda}{\beta}.
\end{equation}
Replacing $t$ by $t+1$ in \eqref{poly}, we obtain
\begin{equation}\label{poly_2}
\Phi_+(\lambda,t+1)=\Big\{\Phi_-(\lambda,t+2)-\frac{1-\alpha}{\lambda}
\Phi_-(\lambda,t+1)\Big\}\frac{\lambda}{\beta}.
\end{equation}
Using successively \eqref{rec2}, \eqref{poly_2} and
\eqref{poly}, we have:
\begin{eqnarray*}
\alpha\lambda\Phi_-(\lambda,t)&=&\Phi_+(\lambda,t+1)-(1-\beta)\lambda\Phi_+(\lambda,t)\\
&=&\frac{\lambda}{\beta}\left\{\Phi_-(\lambda,t+2)-\frac{1-\alpha}{\lambda}\Phi_-(\lambda,t+1)
\right\}\\
&&-\frac{1-\beta}{\beta}\lambda^2\left\{\Phi_-(\lambda,t+1)-\frac{1-\alpha}{\lambda}
\Phi_-(\lambda,t)\right\}.
\end{eqnarray*}
Finally
\[
\Phi_-(\lambda,t+2)-\left(\frac{1-\alpha}{\lambda}+(1-\beta)\lambda
\right)\Phi_-(\lambda,t+1)+\Big((1-\alpha)(1-\beta)-\alpha\beta\Big)\Phi_-(\lambda,t)=0.
\]
The proof concerning $f(\lambda,t)=\Phi_-(\lambda,t)$ is similar and is left to the reader.
\end{proof}
In order to obtain the explicit form of $\Phi_-(\lambda,t)$ and $\Phi_+(\lambda,t)$ 
in terms of $\lambda$ and $t$, it suffices to compute the roots $\theta_-$ and $\theta_+$ 
of the following polynomial
\begin{equation}
\label{polybis}
\theta^2-\left(\frac{1-\alpha}{\lambda}+(1-\beta)\lambda\right)\theta+1-\alpha-\beta=0
\end{equation}
Its discriminant equals
\begin{equation}
\label{5A}
\mathcal{D}=\Big(\frac{1-\alpha}{\lambda}+(1-\beta)\lambda
\Big)^2-4(1-\alpha-\beta).
\end{equation}
It is clear that
\begin{eqnarray}
\label{5AA}
\mathcal{D}&=&\Big(\frac{1-\alpha}{\lambda}+(1-\beta)\lambda+2\sqrt{\rho}\Big)
\Big(\frac{1-\alpha}{\lambda}+(1-\beta)\lambda-2\sqrt{\rho}\Big)\nonumber\\
&=&\frac{1}{\lambda}\Big(\frac{1-\alpha}{\lambda}+(1-\beta)\lambda+2\sqrt{\rho}\Big)
\Big((1-\beta)\lambda^2-2\sqrt{\rho}\lambda+1-\alpha\Big).
\end{eqnarray}
Since the discriminant of $\lambda\to (1-\beta)\lambda^2-2\sqrt{\rho}\lambda+1-\alpha$ 
is equal to $-4\alpha\beta$ then $\mathcal{D}>0$ for any $\lambda>0$.\\
Consequently the roots of \eqref{polybis} are:
\begin{equation}\label{B1}
\theta_\pm=\frac{1}{2}\Big(\frac{1-\alpha}{\lambda}+(1-\beta)\lambda
\pm\sqrt{\mathcal{D}}\Big).
\end{equation}
We deduce the following result.
\begin{prop}\label{a+a-1}
1) The moment generating
function $\Phi(\lambda,t)$ satisfies
\begin{equation}
\label{puiss} \Phi(\lambda,t)=a_+\theta_+^t+a_-\theta_-^t
\end{equation}
with \[
a_+=\frac{1-\alpha+\lambda(\lambda\alpha-\theta_-)}{\lambda^2\sqrt{\mathcal{D}}}
\quad\quad\mbox{and}\quad
a_-=\frac{1}{\lambda}-a_+\quad\mbox{if}\ X_0=Y_0=-1
\]
and \[
a_+=\frac{(1-\beta)\lambda^2+\beta-\lambda\theta_-}{\sqrt{\mathcal{D}}}\quad
\quad\mbox{and}\quad
a_-=\lambda-a_+\quad\mbox{if}\ X_0=Y_0=1.
\]
\end{prop}
\begin{proof} Suppose that $X_0=Y_0=-1$. Let us first determine the values of 
the generating function at time $t=0$ and $t=1$:
\[
\Phi(\lambda,0)=\Phi_+(\lambda,0)+\Phi_-(\lambda,0)=\frac{1}{\lambda}\P(Y_0=-1)
+\lambda\P(Y_0=1)=\frac{1}{\lambda}=a_++a_-
\]
Moreover, using \eqref{rec1} and \eqref{rec2} with $t=0$, we get
\[
\Phi(\lambda,1)=\Phi_+(\lambda,1)+\Phi_-(\lambda,1)=\Big(\frac{1-\alpha}{\lambda}
+\alpha\lambda
\Big)\Phi_-(\lambda,0)=\frac{1-\alpha}{\lambda^2}+\alpha=
a_+\theta_++ a_-\theta_-.
\]
It is clear that Lemma \ref{***3} and $\Phi(\lambda,t)=\Phi_+(\lambda,t)+\Phi_-(\lambda,t)$
implies that $\Phi(\lambda,t)$ satisfies \eqref{seconddegr}. Then \eqref{puiss} follows by
standard arguments. The second case $X_0=Y_0=1$ can be proved in a similar way.\end{proof}
\subsection{Proof of Theorem \ref{conv1}}
\label{***}
We keep the notations given in Section \ref{section_not}. Let $\alpha_0$ and $\beta_0$ 
be two real numbers in $[0,1]$. Let $\Delta_x$ be a small space parameter so that:
\[
0\le \alpha_0+c_0\Delta_x\le 1,\quad 0\le \beta_0+c_1\Delta_x\le 1,
\]
where $c_0$ and $c_1$ belong to $\Rset$.\\
Note that $\alpha_0>0$ (resp. $\beta_0>0$) implies that $\alpha_0+c_0\Delta_x> 0$ 
(resp. $\beta_0+c_1\Delta_x> 0$) when $\Delta_x$ is small enough. If $\alpha_0<1$ 
(resp. $\beta_0<1$), similarly $\alpha_0+c_0\Delta_x< 1$ (resp. $\beta_0+c_1\Delta_x< 1$) 
as soon as $\Delta_x$ is small. In the case $\alpha_0=1$ (resp. $\beta_0=1$) $c_0$ 
(resp. $c_1$) has to be chosen in $]-\infty,0]$.\\
We assume that the coefficients of the transition probability matrix $\pi^\Delta$ 
of the Markov chain $(Y_t)$ satisfy:
\begin{equation}
\label{*3**}
\alpha=\alpha_0+c_0\Delta_x,\quad \beta=\beta_0+c_1\Delta_x
\end{equation}
i.e. $\pi^\Delta$ is given by \eqref{1-4B}.
$(X_t)$ is defined by \eqref{25*1} and $(Z_s^\Delta)$ is the normalized persistent random walk:
\[
 Z^\Delta_s=\Delta_x X_{s/\Delta_t}, \quad (\Delta_t>0, \ \Delta_x>0, \ s\in\Delta_t\Nset).
 \]
$(\tilde{Z}_t^\Delta;\ t\ge 0)$ denotes the linear interpolation of $(Z_t^\Delta)$.\\
Recall that  $\rho_0=1-\alpha_0-\beta_0$
and  $\eta_0=\beta_0-\alpha_0$. Note that $\rho_0\neq 1\Longleftrightarrow \alpha_0+\beta_0\neq 0$
\begin{prop}
\label{conv11} Let $\rho_0\neq1$,\\
1) if $r\Delta_t=\Delta_x$ with $r>0$ then $\tilde{Z}^\Delta_t$ converges 
towards the deterministic limit $-\frac{rt\eta_0}{1-\rho_0}$ as $\Delta_x$ tends to $0$.\\
2) if $r\Delta_t=\Delta_x^2$ with $r>0$,
$\tilde{Z}^\Delta_t+\frac{t\sqrt{r}\eta_0}{(1-\rho_0)\sqrt{\Delta_t}}$ converges in
distribution to the Gaussian law with mean
\begin{equation}
\label{moyenne} m=rt\Big(\frac{-\overline{c}}{1-\rho_0}+\frac{\eta_0
c}{(1-\rho_0)^2} \Big)
\end{equation}
and variance
\begin{equation}\label{variance}
\sigma^2=\frac{r(1+\rho_0)}{1-\rho_0}\Big(1-\frac{\eta_0^2}{(1-\rho_0)^2}\Big)t,
\end{equation}
where
\begin{equation}
\label{encorcoeff} c=c_0+c_1\quad \mbox{and}\quad
\overline{c}=c_1-c_0.
\end{equation}
\end{prop}
\begin{proof} We shall prove the statement under the condition
$X_0=Y_0=-1$. If $X_0=Y_0=+1$, the limit is obtained by changing the sign and 
replacing $c_0$ (resp. $c_1$) by $c_1$ (resp. $c_0$).\\
{\bf 1)} Let $\Phi(\lambda,t)$ be the generating function associated with $X_t$. 
In order to determine the limit distribution of $Z_t^\Delta$, let us introduce:
\begin{equation}
\label{lapla} \phi(\mu,t)=\mathbb{E}_{-1}[e^{-\mu \tilde{Z}_t^\Delta}],
\end{equation}
where $\E_{-1}$ denotes the expectation when $Y_0=-1$. Observe that
\begin{equation}
\label{5.22b} \phi(\mu,t)=\Phi(e^{-\mu\Delta_x},\frac{t}{\Delta_t})=
\mathbb{E}_{-1}[e^{-\mu \Delta_x X(t/\Delta_t)}],
\end{equation}
when $t/\Delta_t\in\Nset$.\\
According to Proposition \ref{a+a-1}, when $t/\Delta_t\in\Nset$, $\phi(\mu,t)$ 
can be expressed in terms of $a_+$, $a_-$ and $\sqrt{\mathcal{D}}$.\\
First let us study the asymptotic expansion of the discriminant $\mathcal{D}$ 
as $\Delta_x\to 0$. It is convenient to set:
 \begin{equation}
 \label{star}
 \bar{\delta}=c_0\Delta_x\quad\mbox{and}\quad \hat{\delta}=c_1\Delta_x.
 \end{equation}
Applying \eqref{5A} with $\alpha=\alpha_0+\bar{\delta}$ and $\beta=\beta_0
+\hat{\delta}$ we have:
\[
\mathcal{D}=\left((1-\alpha_0-\bar{\delta})e^{\mu
\Delta_x}+(1-\beta_0-\hat{\delta})e^{-\mu
\Delta_x}\right)^2-4(1-\alpha_0-\beta_0-\bar{\delta}-\hat{\delta}).
\]
By \eqref{star} we get
\begin{eqnarray}\label{delta}
\mathcal{D}&=\Big(&(2-\alpha_0-\beta_0)+\Delta_x\Big(\mu(\beta_0-\alpha_0)-c_0-c_1\Big)\\
&&+\Delta_x^2\left(\frac{\mu^2}{2}\,(2-\alpha_0-\beta_0)+\mu(c_1-c_0)\right)
+o(\Delta_x^2)\Big)^2\nonumber\\
&-&4\Big(1-\alpha_0-\beta_0-\Delta_x(c_0+c_1)\Big)
\end{eqnarray}
It is clear that $\mathcal{D}$ admits the following asymptotic expansion, as $\Delta_x\to 0$:
\[
\mathcal{D}=A_0+A_1\Delta_x+A_2\Delta_x^2+o(\Delta_x^2)
\]
It is usefull to note that $\alpha_0$ and $\beta_0$ can be expressed in terms of 
$\eta_0$ and $\rho_0$:
\[
\alpha_0=\frac{1-\eta_0-\rho_0}{2}\quad\mbox{and}\quad
\beta_0=\frac{1+\eta_0-\rho_0}{2}.
\]
Let us compute $A_0$, $A_1$ and $A_2$ using standard analysis:
\[
A_0=(2-\alpha_0-\beta_0)^2-4(1-\alpha_0-\beta_0)=\alpha_0^2+\beta_0^2
+2\alpha_0\beta_0=(\alpha_0+\beta_0)^2=(1-\rho_0)^2.
\]
\begin{eqnarray}
\label{5B}
A_1&=&2(2-\alpha_0-\beta_0)\Big(\mu(\beta_0-\alpha_0)-(c_0+c_1)\Big)+4(c_0+c_1)\nonumber\\
&=&2\mu(2-\alpha_0-\beta_0)(\beta_0-\alpha_0)-4(c_0+c_1)+2(\alpha_0+\beta_0)(c_0+c_1)
+4(c_0+c_1)\nonumber\\
&=&2\Big\{\mu(2-\alpha_0-\beta_0)(\beta_0-\alpha_0)+(\alpha_0+\beta_0)(c_0+c_1)\Big\}\nonumber\\
&=&2\Big(\mu\eta_0(1+\rho_0)+c(1-\rho_0)\Big).
\end{eqnarray}
\begin{eqnarray}
\label{5C}
A_2&=&2(2-\alpha_0-\beta_0)\Big(\frac{\mu^2}{2}(2-\alpha_0-\beta_0)
+\mu(c_1-c_0)\Big)+\Big(\mu(\beta_0-\alpha_0)-(c_0+c_1)\Big)^2\nonumber\\
&=&\mu^2\Big((2-\alpha_0-\beta_0)^2+(\beta_0-\alpha_0)^2 \Big)\nonumber\\
&&+2\mu\Big((2-\alpha_0-\beta_0)(c_1-c_0)-(\beta_0-\alpha_0)(c_0+c_1) \Big) +(c_0+c_1)^2\nonumber\\
&=&2\mu^2\Big((\alpha_0-1)^2+(\beta_0-1)^2\Big)+4\mu\Big((1-\beta_0)c_1-(1-\alpha_0)c_0 \Big)
+(c_0+c_1)^2\nonumber\\
&=&\mu^2(\eta_0^2+(1+\rho_0)^2)+2\mu\Big((1+\rho_0)\overline{c}-\eta_0
c\Big)+c^2.
\end{eqnarray}
Under the condition $\rho_0\neq 1$, we have
\[
\sqrt{\mathcal{D}}=(1-\rho_0)\sqrt{1+\frac{A_1}{(1-\rho_0)^2}\Delta_x
+\frac{A_2}{(1-\rho_0)^2}\Delta_x^2+o(\Delta_x^2)}
\]
Hence
\[
\sqrt{\mathcal{D}}=B_0+B_1\Delta_x+B_2\Delta_x^2+o(\Delta_x^2)
\]
with
\[
B_0=1-\rho_0,
\]
\begin{eqnarray*}
B_1&=&\frac{1}{2}\frac{A_1}{1-\rho_0}=\frac{1}{1-\rho_0}
\Big\{\mu(2-\alpha_0-\beta_0)(\beta_0-\alpha_0)+(\alpha_0+\beta_0)(c_0+c_1)\Big\}\\
&=&\mu \,\frac{\eta_0(1+\rho_0)}{1-\rho_0}+c
\end{eqnarray*}
\begin{eqnarray}
\label{5D}
B_2&=&\frac{1}{2}\frac{A_2}{1-\rho_0}-\frac{1}{8}\frac{A_1^2}{(1-\rho_0)^3}.
\end{eqnarray}
As a result, $B_2$ is a second order polynomial function with respect to the 
$\mu$-variable: $$B_2=\mu^2 B_{22}+\mu B_{21}+B_{20}.$$
Identities \eqref{5B}, \eqref{5C} and \eqref{5D} imply:
\[
B_{20}=\frac{c^2}{2(1-\rho_0)}-\frac{\Big(2c(1-\rho_0)\Big)^2}{8(1-\rho_0)^3}=0
\]
\[
B_{21}=\frac{1}{2}\frac{2\Big((1+\rho_0)\overline{c}-\eta_0
c\Big)}{1-\rho_0}-\frac{1}{8}\frac{8\eta_0
c(1-\rho_0)(1+\rho_0)}{(1-\rho_0)^3}=\overline{c}\frac{1+\rho_0}{1-\rho_0}-\frac{2\eta_0
c}{(1-\rho_0)^2}
\]
\begin{eqnarray*}
B_{22}&=&\frac{1}{2}\frac{\eta_0^2+(1+\rho_0)^2}{1-\rho_0}
-\frac{1}{8}\frac{4\eta_0^2(1+\rho_0)^2}{(1-\rho_0)^3}=
\frac{1}{2}\frac{\Big(\eta_0^2+(1+\rho_0)^2\Big)(1-\rho_0)^2
-\eta_0^2(1+\rho_0)^2}{(1-\rho_0)^3}\\
&=&\frac{(1+\rho_0)^2}{2(1-\rho_0)}-\frac{2\eta_0^2\rho_0}{(1-\rho_0)^3}.
\end{eqnarray*}
Consequently \begin{eqnarray}\label{B3}
\sqrt{\mathcal{D}}&=&1-\rho_0+\Big(\mu\,\frac{\eta_0(1+\rho_0)}{1-\rho_0}
+c\Big)\Delta_x+\Big\{\mu^2\Big(\frac{(1+\rho_0)^2}{2(1-\rho_0)}
-\frac{2\eta_0^2\rho_0}{(1-\rho_0)^3}\Big)\nonumber\\
&&+\mu\Big(\overline{c}\,\frac{1+\rho_0}{1-\rho_0}-\frac{2\eta_0
c}{(1-\rho_0)^2}\Big) \Big\}\Delta_x^2+o(\Delta_x^2).
\end{eqnarray}
{\bf 2)} The first order development suffices to determine the limit of 
$\phi(\mu,t)$ as $\Delta_x\to 0$. Indeed
\begin{eqnarray}\label{B2}
\sqrt{\mathcal{D}}&=&1-\rho_0+\frac{\Delta_x}{1-\rho_0}\Big\{\mu\eta_0(1+\rho_0)
+c(1-\rho_0)\Big\}+o(\Delta_x).
\end{eqnarray}
From \eqref{B1} and \eqref{*3**} we can easely deduce
\begin{eqnarray*}
\theta_\pm&=&\frac{1}{2}(1+\rho_0)+\frac{\Delta_x}{2}(\mu\eta_0-c)
\pm\frac{1}{2}\Big\{1-\rho_0+\Delta_x\Big(\frac{\mu\eta_0(1+\rho_0)}{1-\rho_0}+c\Big)\Big\}
+o(\Delta_x).
\end{eqnarray*}
Then
\begin{equation}
\label{R4} \theta_+=1+\Delta_x\frac{\mu\eta_0}{1-\rho_0}+o(\Delta_x)
\quad\mbox{and}\quad\theta_-=\rho_0-\Delta_x\Big(\frac{\mu\eta_0\rho_0}{1-\rho_0}
+c\Big)+o(\Delta_x).
\end{equation}
Let $t'=\lfloor\frac{t}{\Delta_t} \rfloor \Delta_t$. Since
$\tilde{Z}^\Delta_t=\tilde{Z}^\Delta_{t'}+(t-t')\Delta_x Y_{\lfloor t/\Delta_t\rfloor+1}$ 
and $\vert Y_n\vert\le 1$, then
\begin{equation}
\label{5.31b}
\vert \phi(\mu,t)-\phi(\mu,t')\vert\le C\Delta_x\Delta_t,
\end{equation}
where $C$ is a constant which only depends on $\mu$.\\
Recall that identity \eqref{lapla} and Proposition \ref{a+a-1} lead to
\begin{equation}\label{5.27b}
\phi(\mu,t')=a_+\theta_+^{t'/\Delta_t}+a_-\theta_-^{t'/\Delta_t}
\end{equation}
where
\begin{equation}
\label{R3}
a_+=\frac{(1-\alpha)e^{2\mu\Delta_x}+\alpha-\theta_-e^{\mu\Delta_x}}{\sqrt{\mathcal{D}}}
\quad\quad\mbox{and}\quad a_-=e^{\mu\Delta_x}-a_+.
\end{equation}
It is obvious that \eqref{R3} and \eqref{R4} imply: $\lim_{\Delta_x\to 0}a_+=1$
and $\lim_{\Delta_x\to 0}a_-=0$.\\ Since $\lim_{\Delta_t\to
0}\theta_-=\rho_0$ and $-1<\rho_0<1$ then
\begin{equation}
\label{reflim} \lim_{\Delta_x\, \Delta_t\to 0}a_-\theta_-^{t'/\Delta_t}=0.
\end{equation}
Consequently, the second term in \eqref{5.27b} tends to $0$. It is important to 
note that the initial condition $X_0=Y_0=-1$ disappears. Let us study the first 
term in the right hand side of \eqref{5.27b}. Note that $\lim_{\Delta_x\to 0}\theta_+=1$, 
then if $\Delta_x$ is small enough, we can take the logarithm of $\theta_+$. 
From \eqref{R4} a straightforward calculation gives
\[
\log\theta_+=\Delta_x\frac{\mu\eta_0}{1-\rho_0}+o(\Delta_x)
\]
Choosing $r\Delta_t=\Delta_x$ and using \eqref{5.27b}, \eqref{reflim} and \eqref{5.31b}, 
we obtain the following limit:
\[
\lim_{\Delta_x\to 0}\phi(\mu,t)=\exp
\{\frac{r\mu\eta_0 t}{1-\rho_0}\}.
\]
Since the convergence holds for any $\mu\in\Rset$, we can conclude 
(cf Theorem 3 in \cite{curtiss42}) that
\[
\lim_{\Delta_x\to 0}\E_{-1}[\exp (iu \tilde{Z}^\Delta_t)]=\exp
\Big\{-\frac{iu r\eta_0 t}{1-\rho_0} \Big\},\quad\mbox{for any}\ u\in\Rset.
\]
Thus $\tilde{Z}^\Delta_t$ converges in distribution, as $\Delta_x\to 0$, to  
the Dirac measure at  $-\frac{r\eta_0 t}{1-\rho_0}$.\\
{\bf 3)} Next, we consider the convergence of the process
\[
\xi^\Delta_t=\tilde{Z}^\Delta_t+\frac{t\eta_0\sqrt{r}}{(1-\rho_0)\sqrt{\Delta_t}}.
\]
Hence we define
\[
\psi(\mu,t)=\E_{-1}[e^{-\mu \xi^\Delta_t}]=e^{-\frac{\mu
t\eta_0\sqrt{r}}{(1-\rho_0)\sqrt{\Delta_t}}}\phi(\mu,t).
\]
To determine the limit of $\psi(\mu,t)$ as $\Delta_t,\Delta_x\to 0$, 
from \eqref{5.27b} and \eqref{reflim} we may deduce that it suffices to compute 
the second order development of the root $\theta_+$. Using
\eqref{B1} and \eqref{B3} we get:
\begin{eqnarray*}
\theta_+&=&\frac{1}{2}(1+\rho_0)+\frac{\Delta_x}{2}(\mu\eta_0-c)
+\frac{\Delta_x^2}{2}\Big(\frac{\mu^2(1+\rho_0)}{2}+\mu\overline{c} \Big)\\
&+&\frac{1-\rho_0}{2}+\frac{\Delta_x}{2}\Big(\frac{\mu\eta_0(1+\rho_0)}{1-\rho_0}+c \Big)\\
&&+\frac{\Delta_x^2}{2}\Big(\mu^2\Big(\frac{(1+\rho_0)^2}{2(1-\rho_0)}
-\frac{2\eta_0^2\rho_0}{(1-\rho_0)^3}\Big)+\mu\Big(\overline{c}\frac{1+\rho_0}{1-\rho_0}
-\frac{2\eta_0 c}{(1-\rho_0)^2}\Big) \Big)+o(\Delta_x^2).
\end{eqnarray*}
As a result
\begin{equation}\label{***5D}
\theta_+=1+\Delta_x\frac{\mu\eta_0}{1-\rho_0}+\Delta_x^2\Big(\frac{\mu^2}{2}
\Big(\frac{1+\rho_0}{1-\rho_0}-\frac{2\eta_0^2\rho_0}{(1-\rho_0)^3}\Big)
+\mu\Big(\frac{\overline{c}}{1-\rho_0}-\frac{\eta_0
c}{(1-\rho_0)^2} \Big) \Big)+o(\Delta_x^2).
\end{equation}
We take $r\Delta_t=\Delta_x^2$. Then

\begin{eqnarray*}
\lim_{\Delta_x\to 0}\psi(\mu,t)&=&\lim_{\Delta_x\to 0}
\Big(a_+\theta_+^{rt/\Delta_x^2}\exp\Big\{-\frac{\mu r\eta_0 t}{1-\rho_0}
\frac{1}{\Delta_x}\Big\}\Big)\\
&=&\lim_{\Delta_x\to 0}\exp\Big\{-\frac{\mu r\eta_0 t}{1-\rho_0}
\frac{1}{\Delta_x}+\frac{rt}{\Delta_x^2}\log\theta_+\Big\}.
\end{eqnarray*}
It is straightforward to deduce
\begin{equation}
\label{5K}
\lim_{\Delta_x\to 0}\psi(\mu,t)=\exp\Big\{-m\mu+\frac{\sigma^2\mu^2}{2}\Big\}
\end{equation}
with
\begin{equation}
\label{juil1}
m=r\Big(\frac{-\bar{c}}{1-\rho_0}+\frac{\eta_0 c}{(1-\rho_0)^2}\Big)t
\end{equation}
\begin{equation}
\label{juil2}
\sigma^2=\frac{r(1+\rho_0)}{1-\rho_0}\Big(1-\frac{\eta_0^2}{(1-\rho_0)^2}\Big)t.
\end{equation}
{\bf 4)} Since \eqref{5K} holds for any $\mu\in\Rset$, this implies that 
$\xi^\Delta_t$ converges in distribution, as $\Delta_x\to 0$, to the Gaussian 
distribution with mean $m$ and variance $\sigma^2$.
 (see Theorem 3 in \cite{curtiss42})\end{proof}
\begin{prop}
\label{multimarginales} Assume that $\rho_0\neq 1$ and $r\Delta_t=(\Delta_x)^2$. 
Let us denote $\xi^\Delta$ the process defined by
\[
\xi^{\Delta}_t=\tilde{Z}^\Delta_t+\frac{t\sqrt{r}\eta_0}{(1-\rho_0)\sqrt{\Delta_t}}.
\]
Then
$(\xi^{\Delta}_{t_1},\xi_{t_2}^{\Delta},\ldots,\xi^{\Delta}_{t_n})$
converges in distribution, as $\Delta_x\to 0$, towards
$(\xi^0_{t_1},\xi^0_{t_2},\ldots,\xi^0_{t_n})$ where $\xi^0$ is given by
\[
\xi^0_t=r\Big(\frac{-\overline{c}}{1-\rho_0}+\frac{\eta_0 c}{(1-\rho_0)^2}\Big)t
+\sqrt{\frac{r(1+\rho_0)}{1-\rho_0}\Big(1-\frac{\eta_0^2}{(1-\rho_0)^2}\Big)}W_t.
\]
($W_t$, $t\ge 0$) is the one-dimensional Brownian motion starting at $0$.
\end{prop}
\begin{proof} The proof is only presented in the case $n=2$. For simplicity let 
$s=t_1<t_2=t$. 
We are interested in the limit of the random vector $(\xi^\Delta_{s},\xi^\Delta_{t})$. 
Let us then compute the two dimensional Fourier transform
\[
\Psi^\Delta(\mu,\lambda)=\E_{-1}\Big[e^{i\mu(\xi^\Delta_t-\xi^\Delta_s)}
e^{i\lambda\xi^\Delta_s}\Big],\quad (\lambda,\mu\in\Rset).
\]
Since the process $(X_t, Y_t)$ is Markovian, we obtain
\begin{eqnarray*}
\Psi^\Delta(\mu,\lambda)&=&\E_{-1}\Big[e^{i\mu\xi^\Delta_{t-s}}\Big]
\E_{-1}\Big[1_{\{Y(s/\Delta_t)=-1\}}e^{i\lambda\xi^\Delta_s}\Big]\\
&+&\E_{+1}\Big[e^{i\mu\xi^\Delta_{t-s}}\Big]\E_{-1}\Big[1_{\{Y(s/\Delta_t)=+1\}}
e^{i\lambda\xi^\Delta_s}\Big],
\end{eqnarray*}
when $s/\Delta_t$ and $t/\Delta_t$ belongs to $\Nset$.\\
Note that $\vert \xi^\Delta_u-\xi^\Delta_{u'}\vert\le\Delta_x\Delta_t$ 
when $u'=\Big\lfloor\frac{u}{\Delta_t}\Big\rfloor \Delta_t$. Consequently
\begin{eqnarray*}
&&\Psi^\Delta(\mu,\lambda)\underset{\Delta_x\to 0}{\sim}\E_{-1}\Big[e^{i\mu\xi^\Delta_{t'-s'}}\Big]
\E_{-1}\Big[1_{\{Y(s'/\Delta_t)=-1\}}
e^{i\lambda\xi^\Delta_{s'}}\Big]\\
&&+\E_{+1}\Big[e^{i\mu\xi^\Delta_{t'-s'}}\Big]\E_{-1}\Big[1_{\{Y(s'/\Delta_t)=+1\}}
e^{i\lambda\xi^\Delta_{s'}}\Big], \quad(s'=\lfloor s/\Delta_t\rfloor\Delta_t,\ 
t'=\lfloor t/\Delta_t\rfloor \Delta_t).
\end{eqnarray*}
According to Proposition \ref{conv1},
\[
\lim_{\Delta_x\to
0}E_{-1}\Big[e^{i\mu\xi^\Delta_{t'-s'}}\Big]=\lim_{\Delta_x\to
0}E_{+1}\Big[e^{i\mu\xi^\Delta_{t'-s'}}\Big]=e^{(i\mu m-\frac{\sigma^2}{2}\,\mu^2)(t-s)}
\]
where $m$ and $\sigma^2$ are defined by \eqref{juil1}, resp. \eqref{juil2}. 
Then we can deduce:
\begin{eqnarray*}
\lim_{\Delta_x\to
0}\Psi^\Delta(\mu,\lambda)&=&e^{(i\mu m-\frac{\sigma^2}{2}\,\mu^2)(t-s)}
\lim_{\Delta_x\to 0}\E_{-1}\Big[e^{i\lambda\xi^\Delta_{s'}}\Big]\\
&=&e^{(i\mu m-\frac{\sigma^2}{2}\,\mu^2)(t-s)}
\lim_{\Delta_x\to 0}\E_{-1}\Big[e^{i\lambda\xi^\Delta_{s}}\Big]\\
&=&e^{(i\mu m-\frac{\sigma^2}{2}\,\mu^2)(t-s)}e^{(i\lambda m-\frac{\sigma^2}{2}\,\lambda^2)s}\\
&=&\E\Big[\exp\{i\mu(\xi_t^0-\xi_s^0)+i\lambda\xi_s^0\} \Big]
\end{eqnarray*}\end{proof}
We are now able to end the proof of Theorem \ref{conv1} (item 2). We may apply, 
without any change, the measure tension criterium used in the proof of convergence 
of $(Z^\Delta_t)$ in the case $\alpha_0=\beta_0=1$ (see the end of Section \ref{sectionpreuve1}).
This, and Proposition \ref{multimarginales} show that $(\xi_t^\Delta)_{t\ge 0}$ converges 
in distribution as $\Delta_x\to 0$ to the Brownian motion with drift $(\xi_t^0)_{t\ge 0}$. 
\subsection{Proof of Proposition \ref{cas_limit}}
\label{scas_limit}
We suppose $\alpha_0=\beta_0=1$, $c_1=c_0<0$ and $r\Delta_t=\Delta_x^3$ where $r>0$.\\
We briefly sketch the proof of Proposition \ref{cas_limit}. 
The approach is similar to the one developed in the case 2) of Theorem \ref{conv1}. 
We only prove that $\tilde{Z}^\Delta_t$ converges to the Gaussian distribution with 
$0$-mean and variance equals $-rc_0 t$. Using Theorem 3 in \cite{curtiss42}, 
it is equivalent to show
\[
\lim_{\Delta_x\to 0}\E_{-1}\Big[e^{-\mu \tilde{Z}_t^\Delta} \Big]=
e^{\frac{-rc_0 t \mu^2}{2}},\quad \forall \mu\in\Rset.
\]
We have already observed that we may reduce to the case $t/\Delta_t\in\Nset$; 
in this case we have $\tilde{Z}_t^\Delta=Z_t^\Delta$ and
\[
\E_{-1}\Big[ e^{-\mu Z_t^\Delta} \Big]=\Phi\Big(e^{-\mu\Delta_x},\frac{t}{\Delta_t}\Big)
\]
where $\Phi(\lambda,t)$ is the moment generating function associated with $(X_t)$ 
(see the beginning of subsection \ref{mgf}). Recall that $\Phi(\lambda,t)$ is given 
by identity \eqref{puiss}.\\
Note that:
\[
\alpha=\alpha_0+c_0\Delta_x=1+c_0\Delta_x,\quad \beta=\beta_0+c_0\Delta_x=1+c_0\Delta_x.
\]
Since $\alpha$ and $\beta$ have to belong to $[0,1]$, this implies that $c_0<0$.
Recall that $\mathcal{D}$, $\theta_+$ and $\theta_-$ are the real numbers which 
have been defined by \eqref{5A} resp. \eqref{B1} (with $\lambda=e^{-\mu\Delta_x}$). 
We have:
\[
\mathcal{D}=4c_0^2\Delta_x^{2}\cosh^2(\mu\Delta_x)+4(1+2c_0\Delta_x),
\]
\[
\theta_\pm=-c_0\Delta_x\cosh(\mu\Delta_x)\pm\sqrt{c_0^2\Delta_x^{2}\cosh^2(\mu\Delta_x)
+1+2c_0\Delta_x.
}
\]
Using classical analysis we get:
\begin{eqnarray*}
\sqrt{\mathcal{D}}/2&=&\sqrt{1+2c_0\Delta_x+c_0^2\Delta_x^2
+o(\Delta_x^3)}=1+c_0\Delta_x+o(\Delta_x^3),
\end{eqnarray*}
\begin{eqnarray*}
\theta_+
=1-\frac{c_0\mu^2}{2}\Delta_x^3+o(\Delta_x^3),\quad \theta_-=-1-2c_0\Delta_x+o(\Delta_x).
\end{eqnarray*}
\[
\lim_{\Delta_x\to 0}a_+=1,\quad\lim_{\Delta_x\to 0}\theta_+^{t/\Delta_t}=
\lim_{\Delta_x\to 0}\exp\Big\{-\frac{t}{\Delta_t}\,
\frac{c_0\mu^2}{2}\Delta_x^3\Big\}=\exp\Big\{-c_0r\frac{\mu^2}{2}\,t\Big\},
\]
\[
\lim_{\Delta_x\to 0}a_-=0,\quad 
\lim_{\Delta_x\to 0}\vert\theta_-\vert^{t/\Delta_t}=\lim_{\Delta_x\to 0}
\exp\Big\{\frac{t}{\Delta_t}\, 2c_0\Delta_x\Big\}=\lim_{\Delta_x\to 0}
\exp\Big\{\frac{2c_0 rt}{\Delta_x^2}\Big\}=0\quad (c_0<0).
\]

Relation \eqref{puiss} implies that the variable $Z^\Delta_t$ is 
asymptotically normal distributed with variance $-rc_0t$.
{\color{blue}
\ifthenelse{\boolean{complet}}{
\subsection{Further convergence results}
Since we have developed some asymptotic results concerning the moment generating 
function of the normalized persistent walk in the situation $\rho_0\neq 1$, we 
can obtain some supplement informations for the neighbour situation $\rho_0=1$. 
The moment generating function permits to describe the law of $(Z^0_t,\ t\ge 0)$, 
the generalized integrated telegraph noise. This law was already presented in 
Proposition \ref{loi cond} and Proposition \ref{bessel} but only for ITN that 
is $c_0=c_1$ (see the definition of the transition probabilities in the beginning 
of section \ref{results}).
\begin{prop}
\label{cas_weiss} Let $\bar{\delta}=c_0\Delta_x$,
$\hat{\delta}=c_1\Delta_x$, $\alpha_0=\beta_0=0$ and
$r\Delta_t=\Delta_x$.\\ Under the initial condition $Y_0=X_0=-1$, the Laplace 
transform of $Z^\Delta_t$, denoted by $\Phi(\mu,t)$, converges, as $\Delta_x$ 
tends to $0$ towards the following Laplace transform
\begin{eqnarray}\label{transf}
f(\mu,t)&=&\Big(\frac{1}{2}-\frac{2\mu-c}{4\xi}\Big)e^{rt(-\xi-c/2)}
+\Big(\frac{1}{2}+\frac{2\mu+c}{4\xi}\Big)e^{rt(\xi-c/2)}\nonumber\\
&=&e^{-rtc/2}\Big\{\cosh(\xi
rt)+\frac{2\mu+c}{2\xi}\sinh(\xi rt) \Big\}
\end{eqnarray}
where $c=c_0+c_1$, $\overline{c}=c_1-c_0$ and
$\xi=\sqrt{\mu\overline{c}+\mu^2+\frac{c^2}{4}}$.
\end{prop}
\begin{proof} In order to prove the convergence of the Laplace transform
 $\Phi(\mu,t)=\E_{-1}[e^{-\mu Z_t}]$, we use the following decomposition
$\Phi(\mu,t)=a_-\theta_-^{t/\Delta_t}+a_+\theta_+^{t/\Delta_t}$ where
$\theta_-$ and $\theta_+$ are the two roots of the polynomial
\eqref{polybis}. Its discriminant is equal to
\[
\mathcal{D}=R^2-4(1-c\Delta_x)\quad\mbox{with}\quad
R=2-c\Delta_x+\Delta_x^2(\mu\overline{c}+\mu^2)+o(\Delta_x^2)
\]
where $c=c_0+c_1$ and $\overline{c}=c_1-c_0$. Hence, by
\eqref{delta}
\[
\mathcal{D}=4-4c\Delta_x+\Delta_x^2(4\mu\overline{c}+4\mu^2+c^2)-4+4c\Delta_x
+o(\Delta_x^2)=\Delta_x^2(4\mu\overline{c}+4\mu^2+c^2)+o(\Delta_x^2)
\]
The asymptotic developments of the roots are given by
\[
\theta_\pm=\frac{R}{2}\pm\frac{\sqrt{\mathcal{D}}}{2}=1-\frac{c\Delta_x}{2}
\pm\Delta_x\sqrt{\mu\overline{c}+\mu^2+\frac{c^2}{4}}+o(\Delta_x).
\]
This implies
\[
\log\theta_\pm=\Delta_x\Big(\pm\sqrt{\mu\overline{c}+\mu^2+\frac{c^2}{4}}
-\frac{c}{2}\Big)+o(\Delta_x).
\]
In particular, if $\overline{c}=0$ then $c=2c_0$ and
$\theta_\pm=1-c_0\Delta_x\pm\Delta_x\sqrt{\mu^2+c_0^2}+o(\Delta_x)$
and \[
\log\theta_\pm=\Delta_x\Big(\pm\sqrt{\mu^2+c_0^2}-c_0\Big)+o(\Delta_x).
\]
Let us present the development of the prefactor $a_-$ and $a_+$.
We introduce
$\xi=\sqrt{\mu\overline{c}+\mu^2+\frac{c^2}{4}}$. By Lemma
\ref{a+a-1}, under the initial condition $X_0=Y_0=-1$, we get
\begin{eqnarray*}
a_+&=&\frac{(1+2\mu\Delta_x)(1-c_0\Delta_x)+c_0\Delta_x
-(1-\frac{c\Delta_x}{2}-\Delta_x\xi)(1+2\mu\Delta_x)
+o(\Delta_x)}{2\xi\Delta_x+o(\Delta_x)}\\
&=&\frac{c/2+\xi+\mu}{2\xi}+o(\Delta_x)
\end{eqnarray*}
Hence
\[
a_-=1-a_++o(\Delta_x)=\frac{-c/2+\xi-\mu}{2\xi}+o(\Delta_x).
\]
Set $r\Delta_t=\Delta_x$. By letting
$\Delta_x$ tend to $0$, the Laplace transform of $Z_t$ converges 
towards the function $f(\mu,t)$ defined by
\[
f(\mu,t)=\Big(\frac{1}{2}-\frac{2\mu+c}{4\xi}\Big)e^{rt(-\xi-c/2)}
+\Big(\frac{1}{2}+\frac{2\mu+c}{4\xi}\Big)e^{rt(\xi-c/2)}
\]
Other formulation:
\[
f(\mu,t)=e^{-rtc/2}\Big\{\cosh(\xi
rt)+\frac{2\mu+c}{2\xi}\sinh(\xi rt) \Big\}
\]
\end{proof}
\begin{rem}\label{remark} The Laplace transform with respect to 
the time variable can be explicitly computed. We define $F(\mu,s)=\int_0^\infty
e^{-st}f(\mu,t)dt$. In the particular case
$\overline{c}=c_1-c_0=0$, its expression is simplified
\begin{equation}\label{doubletransf}
F(\mu,s)=\frac{s+2rc_0+\mu r}{s^2+2r c_0 s- r^2\mu^2}
\end{equation}
which is to compare with the (double) transform of the solution for the Telegraph equation.\\
Indeed, by \eqref{transf}
\begin{eqnarray*}
F(\mu,s)&=&\Big(\frac{1}{2}-\frac{\mu+c_0}{2\xi} \Big)
\frac{1}{s+r\xi+r c_0}+\Big(\frac{1}{2}+\frac{\mu+c_0}{2\xi} \Big)\frac{1}{s-r\xi+r c_0}\\
&=&\frac{(\xi-\mu-c_0)(s-r\xi+r c_0)+(\xi+\mu+c_0)(s+r\xi+r c_0)}{2\xi((s+rc_0)^2-r^2\xi^2)}\\
&=&\frac{2\xi s+2\xi r c_0+2\mu r\xi+2c_0
r\xi}{2\xi((s+rc_0)^2-r^2\xi^2)}=\frac{ s+ 2r c_0+\mu
r}{(s+rc_0)^2-r^2\xi^2}
\end{eqnarray*}
Since $\xi^2=\mu^2+c_0^2$, we obtain
\[
F(\mu,s)=\frac{s+ 2r c_0+\mu r}{s^2+2src_0-r^2\mu^2}.
\]
\end{rem}
}{}}
\begin{small}
 \bibliographystyle{plain}
 \bibliography{biblio} 
\end{small}
\end{document}

%% file: sdefsf.tex
\newtheorem{thm}{Theorem}[section]
\newtheorem{cor}[thm]{Corollary}
\newtheorem{lem}[thm]{Lemma}

\newtheorem{prop}[thm]{Proposition}

\newtheorem{rem}[thm]{Remark}

\numberwithin{equation}{section}

\newcommand{\Rset}{\mathbb{R}}

\newcommand{\Nset}{\operatorname{I \! N}}

\newcommand{\Zset}{\mathbb{Z}}

\newcommand{\eps}{\varepsilon}

\newcommand{\E}{\operatorname{I \! E}}

\renewcommand{\P}{\operatorname{I \! P}}

\renewcommand{\theta}{\vartheta}
